\documentclass{amsart}
\usepackage[margin=2.2cm]{geometry}
\usepackage{graphics, amsmath,amssymb, enumerate, url, mathtools}
\usepackage{graphicx}
\usepackage{epsfig}
\usepackage{changes}

\newif\ifcolorcomments
\newcommand{\allowcomments}[4]{
\newcommand{#1}[1]{\ifdraft{\ifcolorcomments{\textcolor{#4}{##1 --#3}}\else{\textsl{ ##1 \ --#3}}\fi}\else{}\fi}
}

\colorcommentstrue
\usepackage{amssymb}
\usepackage{amsmath,amsthm}

\usepackage{colortbl}

\newtheorem{theorem}{Theorem}[section]
\newtheorem{lemma}[theorem]{Lemma}
\newtheorem{proposition}[theorem]{Proposition}

\theoremstyle{definition}

\newtheorem{remark}[theorem]{Remark}

\newcommand{\EE}{\mathcal E}

\newcommand{\HH}{\mathcal H}

\newcommand{\N}{\mathbb N}

\newcommand{\Z}{\mathbb Z}

\newcommand{\mbf}{\mathbf}
\newcommand{\0}{\mbf 0}

\renewcommand{\text}{\textup}

\newcommand{\NPC}[1]{\ignorespaces}

\DeclareMathOperator{\Var}{Var}

\newif\ifdraft\drafttrue

\def\N{\mathbb N}
\def\Z{\mathbb Z}

\IfFileExists{marvosym.sty}{
\RequirePackage{marvosym} 
}{

}

\IfFileExists{wasysym.sty}{
\RequirePackage{wasysym}\renewcommand{\emptyset}{{\diameter}}
}{

}

\newcommand*{\myDots}{\ifmmode\mathellipsis\else.\kern-0.07em.\kern-0.07em.\fi}
\DeclarePairedDelimiter\floor{\lfloor}{\rfloor}

\allowcomments{\commumtaz}{MH}{Mumtaz}{green}
\allowcomments{\comnikita}{NS}{Nikita}{blue}
\allowcomments{\combixuan}{BL}{Bixuan}{red}

\newcommand {\ignore}[1] {}

\newcommand{\commenty}[1]{}

\begin{document}

\title[Metrical properties of product of partial quotients with geometric mean]{Metrical properties of the product of partial quotients with geometric mean in continued fractions}

\author[Mumtaz Hussain]{Mumtaz Hussain}
\address{Mumtaz Hussain,  Department of Mathematical and Physical Sciences,  La Trobe University, Bendigo 3552, Australia. }
\email{m.hussain@latrobe.edu.au}

\author[Bixuan Li]{Bixuan Li}
\address{Bixuan Li, School of Mathematics and Statistics, Huazhong University of Science and Technology, Wuhan 430074, P. R. China}
\email{libixuan@hust.edu.cn}

\author{Nikita Shulga}
\address{Nikita Shulga,  Department of Mathematical and Physical Sciences,  La Trobe University, Bendigo 3552, Australia. }
\email{n.shulga@latrobe.edu.au}
\date{}

\maketitle

\numberwithin{equation}{section}

\begin{abstract}
The theory of uniform Diophantine approximation concerns the study of Dirichlet improvable numbers and the metrical aspect of this theory leads to the study of the product of consecutive partial quotients in continued fractions. It is known that the dimension of the set of Dirichlet non-improvable numbers depends upon the number of partial quotients in the product string. However, one can see that the Hausdorff dimension is the same for any number of consecutive partial quotients with a constant gap.  This paper is aimed at a detailed analysis on how the Hausdorff dimension changes when there is a linear gap in indices and the number of partial quotients in the product grows. More precisely, let $d\in \N_{\ge 1}, t\in\Z_{\geq 0}$ and $f(n)=dn+t$, we present the detailed Hausdorff dimension analysis  of the set\begin{equation*}
E_{f}(\psi):=\left\{x\in [0, 1): \sqrt[n]{a_{f(n)}(x)a_{2f(n)}(x)\cdots a_{nf(n)}(x)}\geq \psi(n) \ {\rm for \ infinitely \ many} \ n\in \N\right\}.
\end{equation*}
It is seen that the dimension is larger if $d$ is larger and $t$ has no contribution to the dimension.

%

\end{abstract}

\section{Introduction}

It is well-known that every irrational number $x\in (0, 1)$ has a unique infinite continued fraction expansion. This expansion can be induced from the Gauss map  $T: [0,1)\to [0,1)$ defined by
\[T(0)=0, ~ T(x)=\frac{1}{x}-\floor*{\frac{1}{x}} \textmd{ for }x\in(0,1),\]
where $\lfloor x\rfloor$ denotes the integer part of $x$.  The continued fraction of $x$ is denoted as $x:=[a_{1}(x),a_{2}(x),a_{3}(x),\ldots ]$, with $a_1(x)=\lfloor 1/x \rfloor$ and $a_{n}(x)= a_1(T^{n-1}(x))$ for $n\ge2$, referred to as the partial quotients of $x$.


The starting point of metric Diophantine approximation is Dirichlet theorem, that is the first non-trivial quantitative result. Landmark results of this field are Khintchine's \cite{Khinchin_book} and Jarn\'ik's theorem with Lebesgue and Hausdorff measure. Both of them proved via continued fractions that is closely intertwined with the growth of the large partial quotients.

The classical Borel-Bernstein theorem (1912) states that the Lebesgue measure of the set
\begin{equation*}
\EE_1(\psi):=\left\{x\in [0, 1): a_n(x)\geq \psi(n) \ {\rm for \ infinitely \ many} \ n\in \N\right\}
\end{equation*}
is either zero or full depending upon the convergence or divergence of the series $\sum_{n=1}^\infty \psi(n)^{-1}$ respectively. Here and throughout  $\psi:\N\to [1, \infty)$ will be an arbitrary positive function. Taking $\psi(n)=n$, one can conclude that for almost all $x\in[0,1)$, with respect to Lebesgue measure, $a_n(x)\geq n$ holds for infinitely many $n\in\mathbb{N}$. This implies that the law of large numbers does not hold, that is
\begin{equation}\label{LLN}
\lim_{N\to\infty}\frac{1}{N}\sum_{n<N} a_n(x)=\infty
\end{equation}
for almost all $x$.  Khintchine \cite{Khintchine1935} showed that the weak law of large numbers holds for a suitable normalising sequence, that is, $\sum_{i=1}^na_i(x)/n\log n$ converges to $1/\log 2$ with respect to Lebesgue measure. Philipp \cite{Philipp88} proved that there is no reasonably regular function $\phi:\mathbb N\to\mathbb R_+$ such that
$\sum_{i=1}^na_i(x)/\phi(n)$ almost everywhere converges to a finite nonzero constant. However, Diamond and Vaaler \cite{DiamondVaaler} showed that the strong law of large numbers with the normalising sequence $n\log n$ holds if the largest partial quotient $a_k(x)$ is trimmed from the sum. Inspired by these conclusions, we are interested in studying the growth rate of partial quotients. Firstly, as a metric result, Wang and Wu \cite{WaWu08} gave a Hausdorff dimension of the set $\EE_1(\psi)$ as follows.
\begin{theorem}[Wang-Wu, 2008]
Let $\psi :\mathbb{N}\rightarrow \mathbb [1, \infty)$. Suppose
\begin{equation}\label{psilim}\log B=\liminf\limits_{n\rightarrow \infty }\frac{\log
\psi (n)}{n} \quad { and}\quad \log b=\liminf\limits_{n\rightarrow \infty }\frac{\log
\log \psi (n)}{n}.
\end{equation}
Then
\begin{equation*}
\dim_\HH \EE_1(\psi) =\left\{
\begin{array}{ll}
\inf\Big\{s:\mathsf{P}\left(-s\log\left|T'\right|-s\log B\right)\le 0\Big\}  & {\rm if}\ \ 1\leq B<\infty; \\ [3ex]
\frac{1}{1+b} & {\rm if} \ \  B=\infty.
\end{array}
\right.
\end{equation*}
\end{theorem}
Here and throughout $T^{\prime }$ denotes the derivative of the Gauss map $T$, and $\mathsf{P}$ represents the pressure function defined in the Subsection \ref{Pressure Functions}.

The product of two consecutive partial quotients also plays a significant role in Diophantine approximation. This was first pointed out by Kleinbock and Wadleigh in \cite{KleinbockWadleigh}, they showed that the set of Dirichlet non-improvable numbers can be written in terms of the growth of product of consecutive partial quotients.  Moreover, they proved the Lebesgue measure of set
\begin{equation*}
\EE_2(\psi):=\left\{x\in [0, 1): a_n(x)a_{n+1}(x)\geq \psi(n) \ {\rm for \ infinitely \ many} \ n\in \N\right\}
\end{equation*}
is either full or null depending upon the divergence or convergence of the series $\sum_n\frac{-\log(1-n\psi(n))(1-n\psi(n))}{n}$ respectively. The divergence of (\ref{LLN}) induces the divergence of average of product of partial quotients, that is $$\lim_{N\to\infty}\frac{1}{N}\sum_{n<N} a_n(x)a_{n+1}(x)=\infty$$ for almost all $x$. Hu, Hussain and Yu \cite{HHY} established the weak and strong laws of large numbers with a suitable sequence $n\log^2 n$ for the sums of product of partial quotients. In particular, they
showed that $\sum_{i=1}^na_i(x)a_{i+1}(x)/n\log^2 n$ and $\sum_{i=1}^na_i(x)a_{i+1}(x)-\max_{1\le i\le n}a_i(x)a_{i+1}(x)/n\log^2 n$ converges to $1/(2\log 2)$ with respect to Lebesgue measure.
 A result of Aaronson \cite{Aaronson} states that there is no normalisation $b_n$ such that $\sum_{i=1}^{n}a_i(x)a_{i+1}(x)/b_n$ converges to a nonzero number for almost all $x$. These inspire us to study the growth rate of $a_n(x)a_{n+1}(x)$. With respect to the metric results of product of partial quotients, Hussain, Kleinbock, Wadleigh, and Wang \cite{HKWW} studied the Hausdorff measure of the set $\EE_2(\psi)$, see also \cite{BHS, KSY}.  Huang, Wu and Xu \cite{HuWuXu} calculated the Lebesgue measure and Hausdorff dimension for a natural generalisation of the set $\EE_2(\psi)$.  Namely, for any $m\ge 1$, they considered
\begin{equation*}
\EE_m(\psi):=\left\{x\in [0, 1): a_{n+1}(x)a_{n+2}(x)\cdots a_{n+m}(x)\geq \psi(n) \ {\rm for \ infinitely \ many} \ n\in \N\right\}
\end{equation*}
and proved the following result.

\begin{theorem}[Huang-Wu-Xu, 2020]\label{HWXtheorem}
Let $\psi$ be a positive function and $B$, $b$ are {defined} as in \eqref{psilim}. Then
\begin{equation*}
\dim_\HH \EE_m(\psi) =\left\{
\begin{array}{ll}
\inf \{ s: \mathsf{P}(T, -f_m(s) \log B -s \log | T^{\prime} | ) \leq 0 \} & {\rm if}\ \ 1\leq B<\infty; \\ [3ex]
\frac{1}{1+b} & {\rm if} \ \  B=\infty,
\end{array}
\right.
\end{equation*}
where $f_m(s)$ is given by the following iterative formula:
$$
f_1(s)=s, \,\,\,\,\, f_{k+1}(s) = \frac{s f_k (s)}{1-s+f_k(s)}  \text{ for }k \geq 1.
$$
\end{theorem}
The main idea of estimating the lower bound of the dimension of this set is constructing a suitable subset by analysing the potential coverings. Then, defining a probability measure and using mass distribution principle to estimate the lower bound is a standard process. The specific value of partial quotients in subset grows with an exponential rate. It is worth noting that, the indices of the block $a_{n+1}(x)a_{n+2}(x)\cdots a_{n+m}(x)$  has a constant gap among them. If we set the gap with any other constants, the partial quotients with indices in the gap does not make any contribution for the Hausdorff dimension through the analysis of potential coverings. More precisely, for any constant $c>0$ and $m\ge 1$, the Hausdorff dimension of the set
\begin{equation*}
\EE_{mc}(\psi):=\left\{x\in [0, 1): a_{n+c}(x)a_{n+2c}(x)\cdots a_{n+mc}(x)\geq \psi(n) \ {\rm for \ infinitely \ many} \ n\in \N\right\}
\end{equation*}
is equal to the Hausdorff dimension of the set $\EE_m(\psi)$, since $a_{n+ic+1},\ldots,a_{n+(i+1)c-1}$ for any $1\le i\le m-1$ in the Cantor subset construction is constant. More generally, let $\{c_{\ell}\}_{\ell=1}^{m}$ be any strictly increasing sequence of positive integers, the Cantor subset of the set
\begin{equation*}
\EE_{c_m}(\psi):=\left\{x\in [0, 1): a_{n+c_1}(x)a_{n+c_2}(x)\cdots a_{n+c_m}(x)\geq \psi(n) \ {\rm for \ infinitely \ many} \ n\in \N\right\}
\end{equation*}
is similar to the Cantor subset of $\EE_{mc}$, so that the Hausdorff dimension for both of these sets is the same.  Some results concerning the metrical theory of $\EE_m(\psi)$ can be found in \cite{BBH2, BHS, HuWuXu, HKWW, HLS, KleinbockWadleigh}.

Thus, there are two natural problems:  what kind of gap between the partial quotients can affect the dimension,  and whether the dimension changes or not if the number of partial quotients in the product string is related to $n$. Integrating these two problems together gives rise to the following set. Throughout the paper, our main focus is on a general linear gap in the indices. Let $f(n):=dn+t$ be a linear function, where $d\in\mathbb{N}_{\ge 1}$ and $t\in\mathbb{N}$. Define the set
\begin{equation*}
E_{f}(\psi):=\left\{x\in [0, 1): \sqrt[n]{a_{f(n)}(x)a_{2f(n)}(x)\cdots a_{nf(n)}(x)}\geq \psi(n) \ {\rm for \ infinitely \ many} \ n\in \N\right\}.
\end{equation*}
The main result of the paper is as follows.

\begin{theorem}\label{thm1}
Let $\psi$ be a positive function defined as
\begin{equation}\label{eqnpsi}\log B=\liminf_{n\rightarrow\infty}\frac{\log\psi\left(n\right)}{dn} \quad \text{ and } \quad \log b=\liminf_{n\rightarrow\infty}\frac{\log\log\psi\left(n\right)}{dn^2}.
\end{equation}
Then
\begin{equation*}
\dim_\HH E_f(\psi) =\left\{
\begin{array}{ll}
\inf\Big\{s:\mathsf{P}\left(-s\log\left|T'\right|-\left(2s-1\right)\log B\right)\le 0\Big\} & {\rm if}\ \ 1\leq B<\infty; \\ [3ex]
\frac{1}{1+b} & {\rm if} \ \  B=\infty.
\end{array}
\right.
\end{equation*}
\end{theorem}


%
%

The main part of the proof of the above theorem will be focused on the Hausdorff dimensional of the set
$$E_B:=\left\{x\in [0, 1): a_{f(n)}(x)a_{2f(n)}(x)\cdots a_{nf(n)}(x)\geq B^{dn^2} \ {\rm for \ infinitely \ many} \ n\in \N\right\}.$$ We prove the following result.
\begin{theorem}\label{thm2}For any $1\le B<\infty$, we have
$$\dim_\HH E_B=\inf\Big\{s:\mathsf{P}\left(-s\log\left|T'\right|-\left(2s-1\right)\log B\right)\le 0\Big\}.$$
\end{theorem}
\begin{remark}
Note that if we denote $d \log B =\log \tilde{B}$, then $d$ appears in the pressure function. More precisely, the potential function then would be  $-s\log\left|T'\right|-d^{-1}\left(2s-1\right)\log \tilde{B}$. It can be seen that the dimension is larger if $d$ is larger. In other words, a linear gap in indices of $n$ partial quotients affects the dimension. As we analyzed above, $t$ which is a constant gap does not make any contributions for the Hausdorff dimension. Substituting $d\log b=\log \tilde{b}$, dimension result is transformed into $(1+\tilde{b}^{ d^{-1}})^{-1}$. The conclusion for a larger dimension with a larger linear gap is clearer.
\end{remark}
\subsection{Modus Operandi}The most challenging part of this proof is how to find a suitable subset of $E_B$ for the case $1\leq B<\infty$. For all other cases, the proofs are relatively easier.  In proving the lower bound of Theorem \ref{thm2}, we use the classical method, that is
\begin{itemize}
\item define a good subset of $E_B$; 
\item define a probability measure supported on a suitable subset (Cantor subset) of $E_f(\psi)$;
\item Calculate the H\"older exponent for the probability measure;
\item Use the mass distribution principle to calculate the lower bound of the Hausdorff dimension.
\end{itemize}

Now regarding the second question on whether the dimension changes if the number of partial quotients in the product string {depends} on the index `$n$',  when dealing with a growing number of product of partial quotients as in the set $E_B$, compared with the fixed length of partial quotients as in the sets $\EE_m(\psi)$, $\EE_{mc}(\psi)$, or $\EE_{c_m}(\psi)$, the structure of the set $E_B$ is more complicated with infinite number of variations. More precisely, partial quotients in Cantor subset of \cite{HuWuXu} are of exponential growth, but not in $E_f(\psi)$. The partial quotients $a_{f(n)}$ shows an approximate exponential growth when $n$ is large enough. So that, it is imperative to use this approximate property to define a probability measure supported on a Cantor subset. Then we use the method of mass distribution principle to obtain the lower bound.  More details can be seen in Section 4.

\begin{remark}
It is worth remarking on the well-known  Khintchine's constant (1934)  that states that the  $\lim_{n\to\infty}(a_1\cdots a_n)^{\frac1n}\to 2.685452001\ldots$, independent of $x$, except on a set of Lebesgue measure zero. It is not known yet that whether this constant is irrational or not. This raises a natural, but fundamental question whether the geometric mean of the mixed partial quotients $\lim_{n\to\infty}\left(\prod_{i=1}^{n}a_{if(n)}(x)\right)^{\frac1{n}}$ converges or not, if it does, what is its limit?

\end{remark}

The paper is organised as follows. In Section 2, we introduce some basic properties of continued fractions and some basic facts that we will used later. In Section 3, we will give a sketch of idea of construction of Cantor subset. In Section 4 and 5, we give the proof of Theorem \ref{thm1} and Theorem \ref{thm2} to analyze Hausdorff dimension.

\noindent{\bf Notation.} Throughout this paper, $\mathcal{H}^{s}$ denotes the $s$-dimensional Hausdorff measure, $\dim_\HH$ the Hausdorff dimension and `cl' the closure of a set. We use $a\ll b$ and $c\gg d$ to respectively mean that $0<a/b\le e_{1}$ and $c/d\ge e_{2}>0$ for unspecified constants $e_{1},e_{2}$. Derived from the idea of equal sign, we use the following definition that $m\asymp n$ means $0<m/n\le e_{1}$ and $n/m\ge e_{2}>0$ for unspecified constants $e_{1},e_{2}$.

\medskip

\noindent{\bf Acknowledgements} The research of Mumtaz Hussain and Nikita Shulga is supported by the Australian
Research Council Discovery Project (200100994). Most of this work was carried during Mumtaz and Nikita's stay at the Sydney Mathematical Research Institute (SMRI), we thank their hospitality.  We thank Professors Lingmin Liao and Baowei Wang for useful discussions.

\section{Preliminaries}\label{sec2}
In this section, we will introduce some concepts and properties of continued fractions and pressure function.
\subsection{Continued fractions}
Recall that the $n$th convergents of $x$ are defined as
\begin{align*}
\frac{p_{n}\left(x\right)}{q_{n}\left(x\right)}=\left[a_{1}\left(x\right),\ldots,a_{n}\left(x\right)\right].
\end{align*}
For simplicity, we write $p_{n}\left(a_{1},\ldots,a_{n}\right)=p_{n}$, $q_{n}\left(a_{1},\ldots,a_{n}\right)=q_{n}$ when there is no ambiguity. The rule for constructing these sequences are as follows, for any $k\ge 2$ we have
\begin{equation}\label{rule}
\begin{split}
p_{k}&=a_{k}p_{k-1}+p_{k-2},\\
q_{k}&=a_{k}q_{k-1}+q_{k-2},
\end{split}
\end{equation}
where $p_{0}=0$, $q_{0}=1$.

For any $\left(a_{1},\ldots,a_{n}\right)$ where $a_{i}\in\mathbb{N}$ and $1\le i\le n$, define \textit{a cylinder} $I_{n}\left(a_{1},\ldots,a_{n}\right)$  \textit{of order} $n$ as
\begin{align*}
I_{n}\left(a_{1},\ldots,a_{n}\right):=\left\{x\in\left[0,1\right):a_{1}\left(x\right)=a_{1},\ldots,a_{n}\left(x\right)=a_{n}\right\}.
\end{align*}
Then, we have
\begin{proposition}[{\protect\cite{Khinchin_book}}]\label{range}
For any $n\ge 1$ and $\left(a_{1},\ldots,a_{n}\right)\in\mathbb{N}^{n}$, with $p_{k},q_{k}$ being defined recursively by $\left(\ref{rule}\right)$ where $0\le k\le n$, we have
\begin{equation}I_{n}\left(a_{1},\ldots,a_{n}\right)=\left\{
  \begin{array}{ll}
   \left[\dfrac{p_{n}}{q_{n}},\dfrac{p_{n}+p_{n-1}}{q_{n}+q_{n-1}}\right) &\textrm{\rm if }\ \ n\textrm{ is even},\\[2ex]
   \left(\dfrac{p_{n}+p_{n-1}}{q_{n}+q_{n-1}},\dfrac{p_{n}}{q_{n}}\right] &\textrm{\rm if } \ \ n\textrm{ is odd}.\\
  \end{array}
\right.
\end{equation}
The length of the $n$th cylinder is given by
\begin{equation}\label{length}
\left|I_{n}\left(a_{1},\ldots,a_{n}\right)\right|=\dfrac{1}{q_{n}\left(q_{n}+q_{n-1}\right)}.
\end{equation}
\end{proposition}
For a better estimation, we need the following results. For any $n\ge 1$ and $\left(a_{1},\ldots,a_{n}\right)\in\mathbb{N}^{n}$, we have
\begin{lemma}[{\protect\cite{Khinchin_book}}] For any $1\le k\le n$, we have
\begin{equation}\label{c1}
\dfrac{a_{k}+1}{2}\le\dfrac{q_{n}\left(a_{1},\ldots,a_{n}\right)}{q_{n-1}\left(a_{1},\ldots,a_{k-1},a_{k+1},\ldots,a_{n}\right)}\le a_{k}+1.
\end{equation}
\end{lemma}
\begin{lemma}[{\protect\cite{Khinchin_book}}] For any $k\ge 1$, we have
\begin{equation}\label{c2}
\begin{split}
q_{n+k}\left(a_{1},\ldots,a_{n},a_{n+1},\ldots,a_{n+k}\right)&\ge q_{n}\left(a_{1},\ldots,a_{n}\right)\cdot q_{k}\left(a_{n+1},\ldots,a_{n+k}\right),\\
q_{n+k}\left(a_{1},\ldots,a_{n},a_{n+1},\ldots,a_{n+k}\right)&\le 2q_{n}\left(a_{1},\ldots,a_{n}\right)\cdot q_{k}\left(a_{n+1},\ldots,a_{n+k}\right).
\end{split}
\end{equation}
\end{lemma}
For any $n\ge 1$, since $p_{n-1}q_{n}-p_{n}q_{n-1}=\left(-1\right)^{n}$, combined with simple calculation, then
\begin{equation}\label{qE}
q_{n}\ge 2^{\left(n-1\right)/2}.
\end{equation}

Next, we introduce the mass distribution principle which is useful in calculating the lower bound of Hausdorff dimension.
\begin{proposition}[{\protect\cite{Falconer_book2020}}]\label{MD}Let $E\subseteq\left[0,1\right)$ be a Borel set and $\mu$ be a measure with $\mu\left(E\right)>0$, suppose that for some $s>0$, there is a constant $c>0$ such that for any $x\in\left[0,1\right)$ one has
\begin{equation}\label{MA}
\mu\left(B\left(x,r\right)\right)\le cr^{s},
\end{equation}
where $B\left(x,r\right)$ denotes an open ball centred at $x$ and radius $r$, then $\dim_{\HH}E\ge s$.
\end{proposition}

\subsection{Pressure Function}\label{Pressure Functions}
In this section we present some basic properties of pressure function. The pressure function with a continuous potential can be approximated by the pressure function restricted to the sub-systems in continued fractions. More details of this can be found in Hanus, Mauldin and Urba\'nski \cite{HaMaUr}, Mauldin and Urba\'nski  \cite{MaUr96,MaUr99} or their monogragh \cite{MaUr03}. At this time, we only give the application.

Consider a sub-system $(X_{\mathcal{A}},T)$ of system $([0,1),T)$, define
$$X_{\mathcal{A}}=\Big\{x\in\left[0,1\right):a_{n}\left(x\right)\in\mathcal{A},\textrm{ for all }n\ge 1\Big\},$$
where $\mathcal{A}\subseteq\mathbb{N}$ is finite or infinite. Then, we give the pressure function which is restricted to system $(X_{\mathcal{A}},T)$. Given any real function $\phi:[0,1)\rightarrow\mathbb{R}$, define
\begin{equation}\label{pressure}
\mathsf{P}_{\mathcal{A}}(T,\phi)=\lim_{n\rightarrow\infty}\frac{1}{n}\log \sum_{(a_1,\ldots,a_n)\in\mathcal{A}^{n}}\sup_{x\in X_{\mathcal{A}}} e^{S_n\phi([a_1,\ldots,a_n+x])},
\end{equation}
where $S_n\phi$ denotes the ergodic sum $\phi(x)+\cdots+\phi(T^{n-1}x)$. For simplicity, when $\mathcal{A}=\mathbb{N}$, we denote $\mathsf{P}(T,\phi)$ by $\mathsf{P}_{\mathbb{N}}(T,\phi)$. We will use the notation of $n$th variation of $\phi$ to explain the existence of limit (\ref{pressure}), let
$$\Var_n(\phi):=\sup\Big\{|\phi(x)-\phi(y)|:I_n(x)=I_y(y)\Big\}.$$
In \cite{LiWaWuXu}, authors gave a proposition to guarantee the existence of the limit from the definition of $\mathsf{P}_{\mathcal{A}}(T,\phi)$.
\begin{proposition}[{\protect\cite{LiWaWuXu}}]\label{pro11}
Let $\phi:\left[0,1\right)\rightarrow\mathbb{R}$ be a real function with ${\rm{Var}}_{1}\left(\phi\right)<\infty$ and ${\rm{Var}}_{n}\left(\phi\right)\rightarrow 0$ as $n\rightarrow\infty$. Then the limit defining $\mathsf{P}_{\mathcal{A}}(T,\phi)$ exists and the value of $\mathsf{P}_{\mathcal{A}}(T,\phi)$ remains the same even without taking supremum over $x\in X_{\mathcal{A}}$ in (\ref{pressure}).
\end{proposition}
The following proposition states that the pressure function has a continuity property when the system $([0,1),T)$ is approximated by its sub-system $(X_{\mathcal{A}},T)$.
\begin{proposition}[{\protect\cite{HaMaUr}}]\label{pro}
Let $\phi:\left[0,1\right)\rightarrow\mathbb{R}$ be a real function with ${\rm{Var}}_{1}\left(\phi\right)<\infty$ and ${\rm{Var}}_{n}\left(\phi\right)\rightarrow 0$ as $n\rightarrow\infty$. We have
$$\mathsf{P}_{\mathbb{N}}\left(T,\phi\right)=\sup\Big\{\mathsf{P}_{\mathcal{A}}\left(T,\phi\right):\mathcal{A}\textrm{ is a finite subset of }\mathbb{N}\Big\}.$$
\end{proposition}

Now we specify the potential $\phi$ for our setting. Let
\begin{align*}
\phi(x)=-s\log |T^{\prime}(x)|-(2s-1)\log B.
\end{align*}
Next, denote
$$s_{B}\left(\mathcal{A}\right):=\inf\left\{s\ge 0:\mathsf{P}_\mathcal{A} \left(-s\log\left|T^{\prime}(x)\right|-(2s-1)\log B\right)\le 0\right\},$$
$$s_{B}\left(\mathcal{A},n\right):=\inf\left\{s\ge 0:\sum_{a_{1},\ldots,a_{f(n)}\in\mathcal{A}}\frac{1}{q_{f(n)}^{2s}\left(a_{1},\ldots,a_{f(n)}\right)B^{(2s-1)dn}}\le 1\right\}.$$
If $\mathcal{A}$ is a finite subset of $\mathbb{N}$, we can see that series is equal to 1 when $s=s_{B}\left(\mathcal{A},n\right)$ or that pressure is equal to 0 when $s=s_{B}\left(\mathcal{A}\right)$. For simplicity, when $\mathcal{A}=\mathbb{N}$, we denote $s_{B}\left(\mathcal{A}\right)$ by $s_{B}$ and $s_{B}\left(\mathcal{A},n\right)$ by $s_{B}\left(n\right)$. For some integer $M\in\mathbb{N}$, when $\mathcal{A}=\{1,2,\ldots,M\}$, we denote $s_{B}\left(M,n\right)$ by $s_{B}\left(M\right)$.

It is clear that $\phi$ satisfies the variation condition. By Proposition \ref{pro}, one has
\begin{proposition}\label{pro2}
For any $M\in\mathbb{N}$, we have
\begin{align*}
\lim_{n\rightarrow\infty}s_{B}\left(n\right)=s_{B},\textrm{  }\lim_{n\rightarrow\infty}s_{B}\left(M,n\right)=s_{B}\left(M\right),\textrm{  }\lim_{M\rightarrow\infty}s_{B}\left(M\right)=s_{B}.
\end{align*}
\end{proposition}

The following proposition states the continuity and asymptotic properties of $s_{B}$. The proof of this is similar to Lemma 2.6 in \cite{WaWu08}, we omit the details.
\begin{proposition}
As a function of $B\in\left(1,\infty\right)$, we have $s_{B}$ is continuous and decreasing. Moreover,
$$\lim_{B\rightarrow 1}s_{B}=1,\textrm{  }\lim_{B\rightarrow\infty}s_{B}=\frac{1}{2}.$$
\end{proposition}
In particular, we have $s_{B}\ge\dfrac{1}{2}$.

\section{A suitable subset of $E_B$}
In this section, we mainly focus on how to construct a reasonable good subset of $E_B$ that would act as a Cantor subset. We will prove some lemmas that we will use later. Recall that
$$E_B:=\left\{x\in [0, 1): a_{f(n)}(x)a_{2f(n)}(x)\cdots a_{nf(n)}(x)\geq B^{dn^2} \ {\rm for \ infinitely \ many} \ n\in \N\right\}.$$
For any $n\ge 2$, let $\{A_{i}(n)\}_{1\le i\le n-1}$ be a sequence of real numbers such that $1\le A_i(n)\le B^{dn}$ for any $1\le i\le n-1$, and it will be optimized later. For simplicity, we write
\begin{align*}\alpha_0(n)&=1, \\ \alpha_{i}(n)&=A_{i}(n)\alpha_{i-1}(n) \ \ \text{for any} \ 1\le i\le n-1,\\
w_i&:=(a_{if(n)+1},\ldots,a_{(i+1)f(n)-1}) \ \ \text{for any} \ 0\le i\le n-1.
\end{align*}
Define a set,
\begin{align*}
F(A):=\Big\{x\in\left[0,1\right):& A_{i}(n)^n\le a_{if(n)}(x)\le 2A_{i}(n)^n, 1\le i\le n-1;\\
&\phantom{=\;\;}\left(\frac{B^{dn}}{A_{1}(n)\cdots A_{n-1}(n)}\right)^{n}\le a_{nf(n)}(x)\le 2\left(\frac{B^{dn}}{A_{1}(n)\cdots A_{n-1}(n)}\right)^{n}\textrm{ for i.m. }n\in\mathbb{N}\Big\}.
\end{align*}
There is a natural cover of it, for each $n\ge 1$, let
\begin{align*}
&\quad F_n:=\Big\{x:A_{i}(n)^n\le a_{if(n)}(x)\le 2A_{i}(n)^n, 1\le i\le n-1;\\
&\phantom{=\;\;}\quad\quad\quad\quad\quad\quad\quad\quad\quad\quad\quad\quad\quad\left(\frac{B^{dn}}{A_{1}(n)\cdots A_{n-1}(n)}\right)^{n}\le a_{nf(n)}(x)\le 2\left(\frac{B^{dn}}{A_{1}(n)\cdots A_{n-1}(n)}\right)^{n}\Big\}\\
&=\bigcup_{\substack{a_{if(n)+1},\ldots,a_{(i+1)f(n)-1}\in\mathbb{N},\\0\le i\le n-1}}\Big\{x: a_{if(n)+\ell}(x)=a_{if(n)+\ell}, 0\le i\le n-1,1\le\ell\le f(n)-1; \\
&\phantom{=\;\;}\quad A_{i}(n)^n\le a_{if(n)}(x)\le 2A_{i}(n)^n,1\le i\le n-1; \left(\frac{B^{dn}}{A_{1}(n)\cdots A_{n-1}(n)}\right)^{n}\le a_{nf(n)}(x)\le 2\left(\frac{B^{dn}}{A_{1}(n)\cdots A_{n-1}(n)}\right)^{n}\Big\}\\
&:=\bigcup_{w_i\in\mathbb{N}^{f(n)-1},0\le i\le n-1}F_n(w_0,w_1,\ldots,w_{n-1}).
\end{align*}
Subsequently, we find the most efficient cover of the set $F_n$ in a way that the sequence $\{A_{i}(n)\}_{1\le i\le n-1}$ is optimal, and that determines an effective cover of $F(A)$. More precisely,
\begin{align*}
F(A)&=\bigcap_{N\ge 1}\bigcup_{n\ge N}F_n\\
&=\bigcap_{N\ge 1}\bigcup_{n\ge N}\bigcup_{w_i\in\mathbb{N}^{f(n)-1},0\le i\le n-1}F_n(w_0,w_1,\ldots,w_{n-1}).
\end{align*}
Fix $n\ge 1$, there are $n$ potential covers for $F_n$.
\begin{itemize}
\item Cover type $i$ ($1\le i\le n-1$). For any integers $a_{jf(n)+1},\ldots,a_{(j+1)f(n)-1}\in\mathbb{N}$ with $0\le j\le i-1$, define
$$J_{if(n)-1}(a_1,\ldots,a_{if(n)-1})=\bigcup_{A_{i}(n)^n\le a_{if(n)}\le 2A_{i}(n)^n}I_{if(n)}(a_1,\ldots,a_{if(n)}).$$
The length of this interval is
\begin{align*}
\left|J_{if(n)-1}(a_1,\ldots,a_{if(n)-1})\right|=\sum_{A_{i}(n)^n\le a_{if(n)}\le 2A_{i}(n)^n}\left|\frac{p_{if(n)}}{q_{if(n)}}-\frac{p_{if(n)}+p_{if(n)-1}}{q_{if(n)}+q_{if(n)-1}}\right|\asymp\frac{1}{q_{if(n)-1}^{2}A_{i}(n)^{n}}.
\end{align*}
Thus,
\begin{align*}
\sum_{\substack{a_{jf(n)+1},\ldots,a_{(j+1)f(n)-1}\in\mathbb{N},\\0\le j\le i-1}}\left(\frac{1}{q_{if(n)-1}^{2}A_{i}(n)^{n}}\right)^{s}
&\asymp\left(\sum_{a_1,\ldots,a_{f(n)-1}\in\mathbb{N}}\frac{1}{q_{f(n)-1}^{2s}}\right)^{i}(A_{1}(n)\cdots A_{i-1}(n))^{(1-2s)n}A_i(n)^{-sn}\\
&=\left(\sum_{a_1,\ldots,a_{f(n)-1}\in\mathbb{N}}\frac{1}{q_{f(n)-1}^{2s}}\right)^{i}\alpha_{i-1}(n)^{(1-s)n}\alpha_i(n)^{-sn}.
\end{align*}

\item Cover type $n$. For any integers $a_{jf(n)+1},\ldots,a_{(j+1)f(n)-1}\in\mathbb{N}$ with $0\le j\le n-1$, define
$$J_{nf(n)-1}(a_1,\ldots,a_{nf(n)-1})=\bigcup_{\left(\frac{B^{dn}}{\alpha_{n-1}(n)}\right)^n\le a_{nf(n)}\le 2\left(\frac{B^{dn}}{\alpha_{n-1}(n)}\right)^n}I_{nf(n)}(a_1,\ldots,a_{nf(n)}).$$
The length of this interval is
\begin{align*}
\left|J_{nf(n)-1}(a_1,\ldots,a_{nf(n)-1})\right|&=\sum_{\left(\frac{B^{dn}}{\alpha_{n-1}(n)}\right)^n\le a_{nf(n)}\le 2\left(\frac{B^{dn}}{\alpha_{n-1}(n)}\right)^n}\left|\frac{p_{nf(n)}}{q_{nf(n)}}-\frac{p_{nf(n)}+p_{nf(n)-1}}{q_{nf(n)}+q_{nf(n)-1}}\right|\\
&\asymp\frac{\alpha_{n-1}(n)^n}{q_{nf(n)-1}^{2}B^{dn^2}}.
\end{align*}
With the similar estimates as above, we have
\begin{align*}
\sum_{\substack{a_{jf(n)+1},\ldots,a_{(j+1)f(n)-1}\in\mathbb{N},\\0\le j\le n-1}}\left(\frac{\alpha_{n-2}(n)^{n}}{q_{nf(n)-1}^{2}B^{dn^2}}\right)^{s}
&\asymp\left(\sum_{a_1,\ldots,a_{f(n)-1}\in\mathbb{N}}\frac{1}{q_{f(n)-1}^{2s}}\right)^{n}(A_{1}(n)\cdots A_{n-1}(n))^{(1-2s)n}\frac{\alpha_{n-1}(n)^{sn}}{B^{sdn^2}}\\
&=\left(\sum_{a_1,\ldots,a_{f(n)-1}\in\mathbb{N}}\frac{1}{q_{f(n)-1}^{2s}}\right)^{n}\alpha_{n-1}(n)^{(1-s)n}B^{-sdn^2}.
\end{align*}
\end{itemize}

From the above, for any $n\geq 1$, the most effective cover relies on the minimum of the quantities
\begin{equation*}
\min\Big\{\alpha_{1}(n)^{-s},(\alpha_{1}(n)^{(1-s)}\alpha_{2}(n)^{-s})^{\frac{1}{2}},\ldots,
(\alpha_{n-2}(n)^{(1-s)}\alpha_{n-1}(n)^{-s})^{\frac{1}{n-1}},(\alpha_{n-1}(n)^{(1-s)}B^{-sdn})^{\frac{1}{n}}\Big\}.
\end{equation*}
Given any $\{A_i(n)\}_{i\ge 1,n\ge 1}$ such that $1\le A_i(n)\le B^{dn}$, the quantities $\{\alpha_i(n)\}_{i\ge 1,n\ge 1}$ can be determined accordingly so that $F(A)$ is a subset of $E_B$. Then the dimension of $E_B$ is larger than a value which is determined by $\{\alpha_i(n)\}_{i\ge 1,n\ge 1}$. This leads us to maximize the sequence $\{\alpha_i(n)\}_{1\le i\le n-1}$ to find an optimal cover for each $n\ge 1$. More precisely, we have the following.
\begin{proposition}\label{pro1}
With notation as above, given any $n\ge 2$, the following supremum
\begin{equation}\label{mini}
\sup_{1\le \alpha_{1}(n),\ldots,\alpha_{n-1}(n)\le B^{dn}}\min\Big\{\alpha_{1}(n)^{-s},(\alpha_{1}(n)^{(1-s)}\alpha_{2}(n)^{-s})^{\frac{1}{2}},\ldots,
(\alpha_{n-2}(n)^{(1-s)}\alpha_{n-1}(n)^{-s})^{\frac{1}{n-1}},(\alpha_{n-1}(n)^{(1-s)}B^{-sdn})^{\frac{1}{n}}\Big\}
\end{equation}
is attained when all the terms are equal.
\end{proposition}
Within the proof of this proposition we would need the following lemma from \cite[Lemma 2.12]{HuWuXu}.
\begin{lemma}[{\cite[Lemma 2.12]{HuWuXu}}]
Let $G$ be a subset of $\mathbb{R}^{n}$ and use $G_{x_n}$ to denote the section to the coordinate $x_n$.
\begin{enumerate}[\rm (1)]
\item Let $g(x_1,\ldots,x_n)$ be a function. Then
    $$\sup_{(x_1,\ldots,x_n)\in G}g(x_1,\ldots,x_n)=\sup_{x_n}\sup_{(x_1,\ldots,x_n-1)\in G_{x_n}}g(x_1,\ldots,x_n).$$
    \item $$\min\{a_1,\ldots,a_n\}=\min\{\min\{a_1,\ldots,a_{n-1}\},a_n\}.$$
    \item Let $a(x)$ be a function and $b$ a constant. Then
        $$\sup_{x\in G}\min\{a(x),b\}=\min\{\sup_{x\in G}a(x),b\}.$$
\end{enumerate}
\end{lemma}
Let $\{h_{\ell}\}_{\ell\ge 1}$ be a sequence such that following iterative formula holds, for each $\ell\ge 1$
$$h_{\ell}(s)=\frac{s h_{\ell-1}(s)}{1-s+\ell h_{\ell-1}(s)},\ h_{1}(s)=s.$$
\begin{proof}[Proof of Proposition \ref{pro1}]


We first look at the case when there are only two terms, then the supremum can be expressed as follows. Given $1\le \alpha_{2}(n)\le B^{dn}$, consider
\begin{align*}
\sup_{1\le \alpha_{1}(n)\le B^{dn}}\min\left\{\alpha_{1}(n)^{-s},(\alpha_{1}(n)^{(1-s)}\alpha_2(n)^{-s})^{\frac{1}{2}}\right\}.
\end{align*}
As a function of $\alpha_{1}(n)$, the term $\alpha_{1}(n)^{-s}$ is decreasing and $\alpha_{1}(n)^{(1-s)}$ is increasing, so the supremum is attained when both the terms are equal. More specifically, it is equal to $\alpha_{2}(n)^{-h_{2}(s)}$.

Next, assume that the proposition holds for the first $n-1$ items. Then,
\begin{align*}
&\sup_{1\le \alpha_{1}(n),\ldots,\alpha_{n-1}(n)\le B^{dn}}\min\Big\{\alpha_{1}(n)^{-s},\ldots,
(\alpha_{n-2}(n)^{(1-s)}\alpha_{n-1}(n)^{-s})^{\frac{1}{n-1}},(\alpha_{n-1}(n)^{1-s}B^{-sdn})^{\frac{1}{n}}\Big\}\\
=&\sup_{1\le \alpha_{n-1}(n)\le B^{dn}}\sup_{1\le \alpha_{1}(n),\ldots,\alpha_{n-2}(n)\le B^{dn}}\min\Big\{\min\Big\{\alpha_{1}(n)^{-s},\ldots,
(\alpha_{n-2}(n)^{(1-s)}\alpha_{n-1}(n)^{-s})^{\frac{1}{n-1}}\Big\},(\alpha_{n-1}(n)^{1-s}B^{-sdn} )^{\frac{1}{n}}\Big\}\\
=&\sup_{1\le \alpha_{n-1}(n)\le B^{dn}}\min\Big\{\sup_{1\le \alpha_{1}(n),\ldots,\alpha_{n-2}(n)\le B^{dn}}\min\Big\{\alpha_{1}(n)^{-s},\ldots,
(\alpha_{n-2}(n)^{(1-s)}\alpha_{n-1}(n)^{-s})^{\frac{1}{n-1}}\Big\},(\alpha_{n-1}(n)^{1-s}B^{-sdn} )^{\frac{1}{n}}\Big\}.
\end{align*}
By induction, the supremum
$$\sup_{1\le \alpha_{1}(n),\ldots,\alpha_{n-2}(n)\le B^{dn}}\min\Big\{\alpha_{1}(n)^{-s},\ldots,
(\alpha_{n-2}(n)^{(1-s)}\alpha_{n-1}(n)^{-s})^{\frac{1}{n-1}}\Big\}$$
 is attained when all terms are equal to $\alpha_{n-1}(n)^{-h_{n-1}(s)}$. Note that $\alpha_{n-1}(n)^{-h_{n-1}(s)}$ is decreasing with $a_{n-1}(n)$, but $(\alpha_{n-1}(n)^{1-s}B^{-sdn} )^{\frac{1}{n}}$ is increasing, so that the supremum is attained when all the terms are equal. This completes the proof of the proposition.
\end{proof}

For each $n\ge 2$ and $1\le i\le n-1$, substitute $A_{i}(n)$ into (\ref{mini}) by $A_{i}(n)=\alpha_{i}(n)/\alpha_{i-1}(n)$, and then equate each of them in (\ref{mini}), one induces the following relations. Fix $n\ge 2$, then
\begin{itemize}
\item For any $1\le k\le n-1$,
\begin{equation}\label{cor1}
A_{1}(n)\cdots A_{k}(n)=A_{1}(n)^{-sk}A_{k}(n)^{1-s}(A_{1}(n)\cdots A_{k}(n))^{2s}.
\end{equation}
Indeed, when all terms in the supremum are equal, for any $1\le k\le n-1$, we have
\begin{align*}
&\alpha_{k-1}(n)^{1-s}\alpha_{k}(n)^{-s}=\alpha_{1}(n)^{-sk}\\
\Longleftrightarrow\quad  &(A_1(n)\cdots A_{k-1}(n))^{1-s}(A_1(n)\cdots A_{k}(n))^{-s}=A_{1}(n)^{-sk}\\
\Longleftrightarrow \quad &(A_1(n)\cdots A_{k}(n))^{1-2s}=A_{1}(n)^{-sk}A_{k}(n)^{1-s}.
\end{align*}
\item Each $A_{k}(n)$ for any $1\le k<n-1$ satisfies the following recursive relation
\begin{align}\label{cor2}
s\log A_{k+1}(n)&=s\log A_1(n)+(1-s)\log A_{k}(n).
\end{align}
This is followed by making a difference between the following formulas
\begin{align*}
(1-2s)\log (A_1(n)\cdots A_{k}(n))&=-sk\log A_1(n)+(1-s)\log A_k(n),\\
(1-2s)\log (A_1(n)\cdots A_{k+1}(n))&=-s(k+1)\log A_1(n)+(1-s)\log A_{k+1}(n),
\end{align*}
where $1<k<n-1$.
\item Thus, by (\ref{cor2}), we have
\begin{align}\label{cor3}
\log A_{k}(n)=\left(\frac{1}{2-s^{-1}}-\frac{1}{2-s^{-1}}\left(\frac{1-s}{s}\right)^{k}\right)\log A_{1}(n),
\end{align}
which implies that $A_{1}(n)\le A_{2}(n)\le\ldots\le A_{n-1}(n)$.
\medskip

\item Finally, by taking the limit $n\to\infty$, we conclude
\begin{equation}\label{cor4}
\lim_{n\rightarrow\infty}\log A_{1}(n)=(2-s^{-1})d\log B.
\end{equation}
Letting $A_{n}(n): =B^{dn}/(A_1(n)\cdots A_{n-1}(n))$ we can get the definition of $\alpha_{n}(n)$. Replacing $B^{dn}$ for $A_1(n)\ldots A_n(n)$ in the supremum, we find that (\ref{cor1})--(\ref{cor3}) hold for $A_n(n)$. Thereafter, summing (\ref{cor3}) through for $k=1$ to $n$ and then taking the limit, we obtain the desired result. 
\end{itemize}
The last limit implies that $A_{1}(n)$ is almost a constant when $n$ is large enough. This constant determines the dimension of $E_B$ since the $n$ potential covers are equivalent when $\{A_{i}(n)\}_{1\le i\le n-1}$ is given as above.


\section{Lower bound of $\dim_\HH E_B$}
Let $M,N>2$ be large enough integers such that $s<s_B(M,N)$. For any $0<\varepsilon<1/f(N)$, an estimate of $A_1(n)$ is as follows,
\begin{align}\label{A1}
B^{(2s-1-s\varepsilon)d }\le A_1(n)^{s}\le B^{(2s-1+s\varepsilon)d}\textrm{ for large enough } n.
\end{align}
 For each $k\ge 1$, choose a rapidly increasing sequence of integers $\{r_{k}\}_{k\ge 1}$ such that $r_k^2N\ll \varepsilon r_{k+1}$ and $r_1\gg f(N)^2$, then take $n_k$ such that $f(n_k)-1=r_kf(N)$, where $n_{0}=0$. It can be easily verified that $n_kf(n_{k})\ll \varepsilon f(n_{k+1})$. Let $v_k$ be the largest integer such that $n_{k}f(n_{k})+v_{k}f(N)\le f(n_{k+1})-1$ and denote $m_{k}=n_{k}f(n_{k})+v_{k}f(N)$ for the ease of notation.

Let
\begin{align*}
F:=\Big\{x\in\left[0,1\right)&: A_{i}(n_k)^{n_k}\le a_{if(n_k)}(x)\le 2A_{i}(n_k)^{n_k}, 1\le i\le n_k,\\
&\phantom{=\;\;}a_{m_{k}+1}(x)=\cdots=a_{f(n_{k+1})-1}(x)=2 \textrm { for all }k\ge 1 ;1\le a_{n}\left(x\right)\le M\textrm{ for other cases}\Big\}.
\end{align*}
To simplify the notation, we will make use of a symbolic space described as follows. For each $n\ge 1$,
\begin{align*}
D_n=\Big\{\left(a_{1},\ldots,a_{n}\right)\in\mathbb{N}^{n}&:A_{i}(n_k)^{n_k}\le a_{if(n_k)}\le 2A_{i}(n_k)^{n_k}, 1\le i\le n_k,\\
&\phantom{=\;\;}a_{m_k+1}=\cdots=a_{\min\{f(n_{k})-1,n\}}=2 \textrm { for all }k\ge 1 ,1\le a_{n}\le M\textrm{ for other cases}\Big\}
\end{align*}
and $$D=\bigcup_{n=0}^{\infty}D_{n}\quad \left(D_{0}:=\emptyset\right).$$
Write $I_{0}=\left[0,1\right]$. For any $n\ge 1$ and $\left(a_{1},\ldots,a_{n}\right)\in D_{n}$, we call $I_n\left(a_{1},\ldots,a_{n}\right)$ a $\textit{basic interval of order n}$ and
\begin{equation}\label{J}
J_n\left(a_{1},\ldots,a_{n}\right)=\bigcup_{a_{n+1}}\textrm{cl} I_{n+1}\left(a_{1},\ldots,a_{n+1}\right)
\end{equation}
a $\textit{fundamental interval of order n}$, where the union in $\left(\ref{J}\right)$ is taken over all $a_{n+1}$ such that $\left(a_{1},\ldots,a_{n},a_{n+1}\right)\in~D_{n+1}$.
It is clear that
\begin{equation}
F=\bigcap_{n\ge 1}\bigcup_{\left(a_{1},\ldots,a_{n}\right)\in D_{n}}J_{n}\left(a_{1},\ldots,a_{n}\right).
\end{equation}
We will use $\mathcal{U}$ to denote the following collection of finite words of length $f(N)$:
$$\mathcal{U}=\{w=(\sigma_1,\ldots,\sigma_{f(N)}):1\le\sigma_i\le M,\ 1\le i\le f(N)\}.$$
In the following, we always use $w$ to denote an element from $\mathcal{U}$.
\subsection{Lengths and gaps estimation}
Next, we calculate the lengths of fundamental intervals, which will be split into several cases.
\begin{itemize}
\item When $n=jf(n_k)-1$ for some $k\ge 1$ and $1\le j\le n_{k}$, we have
\begin{align*}
J_{n}\left(a_{1},\ldots,a_{n}\right)=\bigcup_{A_{j}(n_k)^{n_{k}}\le a_{jf(n_{k})}\le 2A_{j}(n_k)^{n_{k}}}I_{n+1}\left(a_{1},\ldots,a_{n},a_{jf(n_k)}\right).
\end{align*}
Therefore,
\begin{align*}
|J_{n}\left(a_{1},\ldots,a_{n}\right)|&=\frac{A_{j}(n_k)^{n_{k}}+1}{((2A_{j}(n_k)^{n_{k}}+1)q_{n}+q_{n-1})(A_{j}(n_k)^{n_{k}}q_{n}+q_{n-1})},
\end{align*}
and
\begin{equation}\label{J1}
\frac{1}{2^{5}A_{j}(n_k)^{n_{k}}q_{n}^{2}}\le |J_{n}(a_{1},\ldots,a_{n})|\le\frac{1}{A_{j}(n_k)^{n_{k}}q_{n}^{2}}.
\end{equation}


\item When $m_k\le n\le f(n_{k})-2$ for some $k\ge 1$, we have
\begin{align*}
J_{n}\left(a_{1},\ldots,a_{n}\right)= I_{n+1}\left(a_{1},\ldots,a_{m_k},2,\ldots,2\right).
\end{align*}
where the number of partial quotients that are equal to 2 is $n-m_k$.
Therefore,
\begin{align*}
\left|J_{n}\left(a_{1},\ldots,a_{n}\right)\right|&=\frac{1}{(2q_n+q_{n-1})(3q_n+q_{n-1})}
\end{align*}
and
\begin{equation}\label{J3}
\frac{1}{2^{4(n-m_k)+5}q_{m_k}^{2}}\le\frac{1}{2^{5}q_{n}^{2}}\le\left|J_{n}\left(a_{1},\ldots,a_{n}\right)\right|\le\frac{1}{2^{2}q_{n}^{2}}\le \frac{1}{2^{2(n-m_k)+2}q_{m_k}^{2}}.
\end{equation}

\item For other cases, we have
\begin{align*}
J_{n}\left(a_{1},\ldots,a_{n}\right)=\bigcup_{1\le a_{n+1}\le M}I_{n+1}\left(a_{1},\ldots, a_{n},a_{n+1}\right).
\end{align*}
Therefore,
\begin{align*}
\left|J_{n}\left(a_{1},\ldots,a_{n}\right)\right|&=\frac{M}{\left(\left(M+1\right)q_{n}+q_{n-1}\right)\left(q_{n}+q_{n-1}\right)}.
\end{align*}
and
\begin{equation}\label{J4}
\frac{1}{6q_{n}^{2}}\le\left|J_{n}\left(a_{1},\ldots,a_{n}\right)\right|\le\frac{1}{q_{n}^{2}}
\end{equation}
\end{itemize}

The gaps between the fundamental intervals can be observed from the position of the intervals in Proposition \ref{range}. Denote $g_n(a_1,\ldots,a_n)$ by the gap between $J_n(a_1,\ldots,a_n)$ and other fundamental intervals of order $n$. We will give a brief explanation that the gap is larger than a constant multiple of $J_n(a_1,\ldots,a_n)$. A detailed proof can be found in \cite{WaWu08}.

When $n=jf(n_k)-1$ and $m_k\le n\le f(n_{k})-2$ for some $k\ge 1$ and $1\le j\le n_{k}$, $J_n$ lies in the middle of $I_n$, then the gap must be larger than the distance between right/left endpoint of $J_n$ and $I_n$. However, this distance only differs by a constant multiple from $|J_n|$. As for other cases, $J_n$ lies in the rightmost/leftmost of $I_n$ since $a_{n+1}$ goes through $1$ to $M$. So the gap is larger than the distance between left/right endpoint of $J_n$ and $I_n$. It implies that
\begin{equation}\label{gap}
g_n(a_1,\ldots,a_n)\ge \frac{1}{M}|J_n(a_1,\ldots,a_n)|.
\end{equation}

\subsection{Mass distribution}

Next, we define a measure $\mu$ supported on $F$. Recall the definition of $\mathcal{U}$. Every element $x\in F$ can be written as the form
\begin{align*}
x=[w_1^{(1,0)},\ldots,w_{r_1}^{(1,0)},&a_{f(n_1)},\ldots,w_1^{(1,n_1-1)},\ldots,w_{r_1}^{(1,n_1-1)},a_{n_1f(n_1)},w_1^{(1,n_1)},\ldots,w_{v_1}^{(1,n_1)},2,\ldots,2\\
&a_{f(n_k)},\ldots,w_1^{(k,n_k-1)},\ldots,w_{r_k}^{(k,n_k-1)},a_{n_kf(n_k)},w_1^{(k,n_k)},\ldots,w_{v_k}^{(k,n_k)},2,\ldots,2,\ldots]
\end{align*}
where $w_{i}^{(k,j)}\in\mathcal{U}$ for all $1\le i\le r_k,0\le j\le n_k,k\ge 1$ and $1\le i\le v_k,j=n_k,k\ge 1$, and
$$A_i(n_{\ell})^{n_{\ell}}\le a_{f(n_{\ell})}\le 2A_i(n_{\ell})^{n_{\ell}} \ \textrm{for all} \ \ell\ge 1.$$
For each $k\ge 1$ and $1\le i\le n_k$, the mass on fundamental intervals of order $if(n_k)$ depends on the number of possible values of $a_{if(n_k)}$ in the subset. In other words, we use counting measure on these fundamental intervals. For intervals of order $n$ with $m_k+1\le n\le f(n_{k+1})-1$, the mass is the  same as the interval of order $n-1$ which contains $J_n$. For simplicity, throughout this paper, we write fundamental interval of order $n$ as $J_n$ rather than $J_n(a_1,\ldots,a_n)$ without ambiguity. Furthermore, uniform distribution and consistency of measure are used to define on the others. More precisely, for integers $N,M>2$ that we chose above, denote
\begin{align*}
u_{N}:=\sum_{w\in\mathcal{U}}\frac{1}{q_{f(N)}^{2s}(w)B^{\left(2s-1\right)dN}}.
\end{align*}
For any $s<s_{B}\left(M,N\right)$, one has $u_N>1$. We define mass by induction. 
\begin{itemize}
\item Let $n\le f(n_2)-1$,
\begin{itemize}
\item For each $1\le i\le r_1$, define
\begin{align*}
\mu(J_{if(N)})=\prod_{t=1}^{i}\frac{1}{u_{N}}\cdot\frac{1}{q_{f(N)}^{2s}(w_t^{(1,0)})B^{\left(2s-1\right)dN}};
\end{align*}

\item For each $1\le j\le n_1$, define
\begin{align*}
\mu(J_{jf(n_{1})})&=\frac{1}{A_{j}(n_1)^{n_1}}\prod_{i=1}^{j-1}\frac{1}{A_{i}(n_1)^{n_{1}}}\prod_{t=1}^{r_{1}}\frac{1}{u_{N}}\cdot\frac{1}{q_{f(N)}^{2s}(w_t^{(1,i)})B^{(2s-1)dN}}\cdot\mu(J_{f(n_1)-1}); 
\end{align*}

\item For each $1\le j\le n_{1}$, when $n=jf(n_{1})+if(N)$, define
\begin{align*}
\mu(J_{n})&=\mu(J_{jf(n_{1})}) \cdot\prod_{t=1}^{i}\frac{1}{u_N}\cdot\frac{1}{q_{f(N)}^{2s}(w_t^{(1,j)})B^{(2s-1)dN}} \ \textrm{for all } 1\le i\le r_1, 1\le j\le n_1-1;\\
\mu(J_{n})&=\mu(J_{n_1f(n_{1})}) \cdot\prod_{t=1}^{i}\frac{1}{u_N}\cdot\frac{1}{q_{f(N)}^{2s}(w_t^{(1,n_1)})B^{(2s-1)dN}}\ \textrm{for all } 1\le i\le v_1, j=n_1;
\end{align*}

\item For each $m_{1}+1\le n\le f(n_{2})-1$, define
\begin{align*}
\mu(J_{n})=\mu(J_{m_{1}});
\end{align*}

\item The mass of other fundamental intervals of order less than $f(n_2)-1$ is designed by the consistency of a measure. For each integer $n$ with $jf(n_1)+(i-1)f(N)<n<jf(n_1)+i f(N)$ for some $0\le j\le n_1$ and $1\le i\le r_1$, define
    \begin{align*}
    \mu(J_n)=\sum_{J_{jf(n_1)+i f(N)}\subseteq\ J_n}\mu(J_{jf(n_1)+i f(N)});
    \end{align*}
\end{itemize}
\item Let $f(n_{k})-1<n\le f(n_{k+1})-1$ for some $k\ge 1$, assume the mass of all the fundamental intervals of order less than $f(n_{k})-1$ has been defined. Then we define mass as follows.
\begin{itemize}
\item For each $1\le j\le n_{k}$, define
\begin{align*}
\mu(J_{jf(n_{k})})&=\frac{1}{A_{j}(n_k)^{n_k}}\prod_{i=1}^{j-1}\frac{1}{A_{i}(n_k)^{n_{k}}}\prod_{t=1}^{r_{k}}\frac{1}{u_{N}}\cdot\frac{1}{q_{f(N)}^{2s}(w_t^{(k,i)})B^{(2s-1)dN}}\cdot\mu(J_{f(n_k)-1}); 
\end{align*}


\item For each $1\le j\le n_{k}$, when $n=jf(n_{k})+if(N)$, define
\begin{align*}
\mu(J_{n})&=\mu(J_{jf(n_{k})}) \cdot\prod_{t=1}^{i}\frac{1}{u_N}\cdot\frac{1}{q_{f(N)}^{2s}(w_t^{(k,j)})B^{(2s-1)dN}}\ \textrm{for all } 1\le i\le r_k, 1\le j\le n_k-1;\\
\mu(J_{n})&=\mu(J_{n_kf(n_{k})}) \cdot\prod_{t=1}^{i}\frac{1}{u_N}\cdot\frac{1}{q_{f(N)}^{2s}(w_t^{(k,n_k)})B^{(2s-1)dN}}\ \textrm{for all } 1\le i\le v_k, j=n_k;
\end{align*}
\item For each $m_{k}+1\le n\le f(n_{k+1})-1$, define
\begin{align*}
\mu(J_{n})=\mu(J_{m_{k}});
\end{align*}

\item The mass of the fundamental intervals of other orders is defined as the summation of the mass of its offsprings to ensure the consistency of measure.

  \end{itemize}
\end{itemize}

\subsection{H\"older exponent of the measure on fundamental intervals}
Following these, we estimate the mass on fundamental intervals that will frequently use properties \eqref{c1}--\eqref{qE} of $\{q_n\}_{n\ge 1}$ in the process.
\begin{itemize}
\item When $n=if(N)$ for some $1\le i< r_{1}$, by the estimation of length of $J_n$ in (\ref{J4}), we have
\begin{equation}\label{E1}
\begin{split}
\mu(J_{if(N)})&\le\prod_{t=1}^{i}\frac{1}{q_{f(N)}^{2s}(w_t^{(1,0)})B^{(2s-1)dN}}\le 2^{2\left(i-1\right)}\cdot\frac{1}{q_{if(N)}^{2s}}\\
&\le\left(\frac{1}{q_{if(N)}^{2}}\right)^{s-\frac{2}{f(N)}}\le 6 |J_{if(N)}|^{s-\frac{2}{f(N)}},
\end{split}
\end{equation}
where the third inequality holds by following
$$q_{if(N)}^2\ge 2^{if(N)-1}\Longrightarrow q_{if(N)}^{2\cdot\frac{2}{f(N)}}\ge 2^{2(i-1)}.$$

\item When $n=f(n_1)-1$, this case will be somewhat different from the previous one.
\begin{equation*}
\begin{split}
\mu(J_{f(n_1)-1})&\le\prod_{t=1}^{r_1}\frac{1}{q_{f(N)}^{2s}(w_t^{(1,0)})B^{(2s-1)dN}}\le 2^{2\left(r_1-1\right)}\cdot\frac{1}{q_{f(n_1)-1}^{2s}B^{(2s-1)r_1dN}}\\
&\le \left(\frac{1}{q_{f(n_1)-1}^{2}}\right)^{s-\frac{2}{f(N)}}\cdot \frac{1}{B^{(2s-1)r_1dN}}.
\end{split}
\end{equation*}
Recall inequalities (\ref{J1}) and (\ref{A1}) and note that when $N$ is large enough it implies that $n_1$ is also large. So, we have
$$|J_{f(n_1)-1}|^s\ge \frac{1}{2^5 A_1(n_1)^{sn_1}q_{f(n_1)-1}^{2s}}\ge\frac{1}{2^{5} B^{dn_1(2s-1+s\varepsilon)}q_{f(n_1)-1}^{2s}}.$$
Thus,
\begin{equation}\label{E2}
\begin{split}
\mu(J_{f(n_1)-1})&\le \left(\frac{1}{q_{f(n_1)-1}^{2}}\right)^{s-\frac{2}{f(N)}}\cdot \frac{1}{B^{(2s-1)(n_1d-r_1t-1)}}\\
&\le 2^5B^{sn_1d\varepsilon+(2s-1)(r_1t+1)}\cdot|J_{f(n_1)-1}|^{s-\frac{2}{f(N)}}\\
&\ll 2^5|J_{f(n_1)-1}|^{s-\frac{2}{f(N)}-\frac{6(t+1)}{f(N)}\cdot\frac{\log B}{\log 2}}.
\end{split}
\end{equation}
The last inequality holds since, when $N$ is large enough, we have
$$B^{sn_1d\varepsilon+2(2s-1)r_1t}\le B^{n_1d\frac{1}{f(N)}+2\frac{f(n_1)-1}{f(N)}t}= B^{\frac{(2t+1)dn_1+2t^2-2t}{f(N)}}\le 2^{\frac{\log B}{\log 2}\cdot\frac{3(t+1)dn_1}{f(N)}}\textrm{ \ and \ } q_{f(n_1)-1}^2\ge 2^{\frac{dn_1}{2}}.$$
\item When $n=jf(n_1)$ for some $1\le j\le n_1$, we have
\begin{align*}
\mu(J_{jf(n_{1})})&\le\frac{1}{A_{j}(n_1)^{n_1}}\prod_{i=1}^{j-1}\frac{1}{A_{i}(n_1)^{n_{1}}}\cdot\prod_{t=1}^{r_{1}}\frac{1}{q_{f(N)}^{2s}(w_t^{(1,i)})B^{(2s-1)dN}}\cdot\mu(J_{f(n_1)-1})\\
&\le \frac{1}{(A_1(n_1)\cdots A_j(n_1))^{n_1}}\prod_{i=0}^{j-1}\left(\frac{1}{q_{f(n_1)-1}^2(w_1^{(1,i)},\ldots,w_{r_1}^{(1,i)})}\right)^{s-\frac{2}{f(N)}}\cdot \frac{1}{B^{(2s-1)jdr_1N}}.
\end{align*}
Observe that the length of $J_{jf(n_{1})}$ is as follows,
\begin{align*}
|J_{jf(n_{1})}|&\ge \frac{1}{6q_{jf(n_1)}^{2}}\ge \frac{1}{6\cdot 2^{6j-2}(A_1(n_1)\cdots A_j(n_1))^{2n_1}}\prod_{i=0}^{j-1}\left(\frac{1}{q_{f(n_1)-1}^2(w_1^{(1,i)},\ldots,w_{r_1}^{(1,i)})}\right).
\end{align*}
Furthermore, we will use the  following inequalities. For any $\ell\ge 1$,
\begin{equation}\label{E31}
\begin{split}
q_{f(n_{\ell})-1}^2\ge 2^{f(n_{\ell})-2}\ge 2^{\frac{f(N)}{2}}\quad&\Longrightarrow\quad\prod_{i=0}^{j-1}q_{f(n_{\ell})-1}^2(w_1^{(\ell,i)},\ldots,w_{r_1}^{(\ell,i)})\ge 2^{\frac{jf(N)}{2}}\\
&\Longrightarrow\quad \prod_{i=0}^{j-1}q_{f(n_{\ell})-1}^{2\cdot\frac{12}{f(N)}}(w_1^{(\ell,i)},\ldots,w_{r_1}^{(\ell,i)})\ge 2^{6j},
\end{split}
\end{equation}
and
\begin{equation}\label{E32}
\begin{split}
q_{f(n_{\ell})-1}^2\ge 2^{dn_{\ell}+t-2}
\quad&\Longrightarrow\quad\prod_{i=0}^{j-1}q_{f(n_{\ell})-1}^{2\cdot\frac{1}{f(N)}}(w_1^{(\ell,i)},\ldots,w_{r_1}^{(\ell,i)})\ge 2^{\frac{dn_{\ell}j}{2f(N)}}\\
&\Longrightarrow\quad \prod_{i=0}^{j-1}q_{f(n_{\ell})-1}^{2\frac{\log B}{\log 2}\cdot\frac{6(t+1)}{f(N)}}(w_1^{(\ell,i)},\ldots,w_{r_1}^{(\ell,i)})\ge 2^{\frac{\log B}{\log 2}\cdot\frac{3dn_{\ell}j(t+1)}{f(N)}}.
\end{split}
\end{equation}
Thus, combined with (\ref{cor1}) and (\ref{A1}), we have  
\begin{equation}\label{E3}
\begin{split}
\mu(J_{jf(n_{1})})&\le \frac{A_1(n_1)^{sjn_1}}{A_{j}(n_1)^{(1-s)n_1}(A_1(n_1)\cdots A_j(n_1))^{2sn_1}}\cdot \prod_{i=0}^{j-1}\left(\frac{1}{q_{f(n_1)-1}^2(w_1^{(1,i)},\ldots,w_{r_1}^{(1,i)})}\right)^{s-\frac{2}{f(N)}}\frac{1}{B^{(2s-1)jdr_1N}}\\
&\le\frac{B^{jn_1(2s-1+s\varepsilon)d}}{(A_1(n_1)\cdots A_j(n_1))^{2sn_1}B^{(2s-1)jdr_1N}}\cdot\prod_{i=0}^{j-1}\left(\frac{1}{q_{f(n_1)-1}^2(w_1^{(1,i)},\ldots,w_{r_1}^{(1,i)})}\right)^{s-\frac{2}{f(N)}}\\
&\le6\cdot2^{6j-2}B^{jdn_1(2s-1+s\varepsilon)-(2s-1)j(dn_1+t-1-r_1t)}|J_{jf(n_1)}|^{s-\frac{2}{f(N)}}\\
&\le 6\cdot2^{6j-2}2^{\frac{\log B}{\log 2}\frac{3(t+1)dn_1j}{f(N)}}|J_{jf(n_1)}|^{s-\frac{2}{f(N)}}\\
&\le 6|J_{jf(n_1)}|^{s-\frac{15}{f(N)}-\frac{\log B}{\log 2}\cdot\frac{6(t+1)}{f(N)}}.
\end{split}
\end{equation}


\item When $n=jf(n_1)-1$ for some $1< j\le n_1$, recall the length of $J_{jf(n_1)-1}$ which is
\begin{align*}
|J_{jf(n_1)-1}|&\ge\frac{1}{2^{5}A_{j}(n_1)^{n_{1}}q_{jf(n_1)-1}^{2}}\\
&\ge\frac{1}{2^{6j-1}A_{j}(n_1)^{n_1}(A_1(n_1)\cdots A_{j-1}(n_1))^{2n_1}}\prod_{i=0}^{j-1}\left(\frac{1}{q_{f(n_1)-1}^2(w_1^{(1,i)},\ldots,w_{r_1}^{(1,i)})}\right).
\end{align*}
For any $1<j\le n_1$, by (\ref{cor2}), one has
$$A_{j}(n_1)^s=A_{1}(n_1)^sA_{j-1}(n_1)^{1-s}.$$
The above equality when combined with (\ref{cor1}) and (\ref{A1}), yields
\begin{equation}\label{E5}
\begin{split}
\mu(J_{jf(n_1)-1})&\le\frac{1}{(A_1(n_1)\cdots A_{j-1}(n_1))^{n_1}}\cdot \prod_{i=0}^{j-1}\left(\frac{1}{q_{f(n_1)-1}^2(w_1^{(1,i)},\ldots,w_{r_1}^{(1,i)})}\right)^{s-\frac{2}{f(N)}}
\frac{1}{B^{(2s-1)jdr_1N}}\\
&\le\frac{A_1(n_1)^{s(j-1)n_1}}{A_{j-1}(n_1)^{(1-s)n_1}(A_1(n_1)\cdots A_{j-1}(n_1))^{2sn_1}}\\
&\quad\quad\quad\quad\quad\quad\quad\cdot\frac{1}{B^{(2s-1)j(dn_1+t-1-r_1t)}}\cdot\prod_{i=0}^{j-1}\left(\frac{1}{q_{f(n_1)-1}^2(w_1^{(1,i)},\ldots,w_{r_1}^{(1,i)})}\right)^{s-\frac{2}{f(N)}}\\
&\le2^{6j-1}\frac{A_1(n_1)^{s(j-1)n_1}A_j(n_1)^{sn_1}B^{j(r_1t-t+1)}}{A_{j-1}(n_1)^{(1-s)n_1}B^{(2s-1)jdn_1}} |J_{jf(n_1)-1}|^{s-\frac{2}{f(N)}}\\
&\le2^{6j-1}\frac{A_1(n_1)^{sjn_1}B^{j(r_1t-t+1)}}{B^{(2s-1)jdn_1}}|J_{jf(n_1)-1}|^{s-\frac{2}{f(N)}}\\
&\le2^{6j-1}B^{\frac{3d(t+1)jn_1}{f(N)}}|J_{jf(n_1)-1}|^{s-\frac{2}{f(N)}}\\
&\le|J_{jf(n_1)-1}|^{s-\frac{15}{f(N)}-\frac{\log B}{\log 2}\cdot\frac{6(t+1)}{f(N)}},
\end{split}
\end{equation}
where the last inequality holds by (\ref{E31}) and (\ref{E32}).


\item When $n=jf(n_1)+if(N)$ for some $1\le j< n_1$ and $1\le i<r_1$, we have
\begin{equation}\label{E7}
\begin{split}
\mu(J_{n})&\le \prod_{t=1}^{i}\frac{1}{q_{f(N)}^{2s}(w_t^{(1,j)})B^{(2s-1)dN}}\cdot\mu(J_{jf(n_1)})\\
&\le \left(\frac{1}{q_{if(N)}^2(w_1^{(1,j)},\ldots,w_{i}^{(1,j)})}\right)^{s-\frac{2}{f(N)}}\cdot\left(\frac{1}{q_{jf(n_1)}^2}\right)^{s-\frac{15}{f(N)}-\frac{\log B}{\log 2}\cdot\frac{6(t+1)}{f(N)}}\\
&\le2\left(\frac{1}{q_{n}^2}\right)^{s-\frac{15}{f(N)}-\frac{\log B}{\log 2}\cdot\frac{6(t+1)}{f(N)}}\\
&\le 12|J_n|^{s-\frac{15}{f(N)}-\frac{\log B}{\log 2}\cdot\frac{6(t+1)}{f(N)}}.
\end{split}
\end{equation}
When $n=n_1f(n_1)+if(N)$ with $1\le i\le v_1$, the process is similar and the H\"older exponent is the same. We omit the details.

\item When $m_1+1\le n< f(n_2)-1$,  note that $n-m_1$ is strictly less than $f(N)$ since $v_1$ is the largest number that we cannot insert into $f(N)$ partial quotients with index between $n_1f(n_1)$ and $f(n_2)-1$. We will use this fact below. Firstly,  observe that
    $$n-1\ge n_1f(n_1)-1\ge\frac{1}{2}n_1r_1f(N)\ \textrm{ and }\ 2d(t+1)n_1\ge dn_1+t-1=r_1f(N)\ge f(N),$$
which imply that
$$n-1\ge\frac{f(N)^2}{4d(t+1)}.$$
Then
$$q_n^{2\cdot\frac{16d(t+1)}{f(N)}}\ge 2^{\frac{16d(t+1)}{f(N)}(n-1)}\ge 2^{4f(N)}\ge 2^{4(n-m_1)}.$$
By using this inequality, we have
\begin{equation}\label{E8}
\begin{split}
\mu(J_{n})&=\mu(J_{m_1})\le \prod_{t=1}^{v_1}\frac{1}{q_{f(N)}^{2s}(w_t^{(1,n_1)})B^{(2s-1)dN}}\cdot\mu(J_{n_1f(n_1)})\\
&\le \left(\frac{1}{q_{v_1f(N)}^2(w_1^{(1,n_1)},\ldots,w_{v_1}^{(1,n_1)})}\right)^{s-\frac{2}{f(N)}}\cdot\left(\frac{1}{q_{n_1f(n_1)}^2}\right)^{s-\frac{15}{f(N)}-\frac{\log B}{\log 2}\cdot\frac{6(t+1)}{f(N)}}\\
&\le2\left(\frac{1}{q_{m_1}^2}\right)^{s-\frac{15}{f(N)}-\frac{\log B}{\log 2}\cdot\frac{6(t+1)}{f(N)}}\\
&\le2^{4(n-m_1)+6}\left(\frac{1}{q_{n}^2}\right)^{s-\frac{15}{f(N)}-\frac{\log B}{\log 2}\cdot\frac{6(t+1)}{f(N)}}\\
&\le 6\cdot2^6|J_n|^{s-\frac{15}{f(N)}-\frac{\log B}{\log 2}\cdot\frac{6(t+1)}{f(N)}-\frac{16d(t+1)}{f(N)}}.
\end{split}
\end{equation}

\item For the remaining cases when $n<f(n_2)-1$, one case is that there are $1\le i\le r_1$ and $0\le j< n_1$ such that $$jf(n_1)+(i-1)f(N)<n<jf(n_1)+i f(N).$$ Next we compare the length of $J_n$ with the length of  $J_{jf(n_{1})+(i-1)f(N)}$  as
\begin{align*}
\frac{\left|J_{n}\right|}{\left|J_{jf(n_1)+(i-1)f(N)+1}\right|}&\ge \frac{1}{6}\frac{q_{jf(n_1)+(i-1)f(N)+1}^{2}}{q_{n}^{2}}\\
&\ge\frac{1}{6}\left(\frac{1}{2^{n-jf(n_1)-(i-1)f(N)-1}a_{jf(n_1)+(i-1)f(N)+2}\cdots a_{n}}\right)^{2}\\ &\ge\frac{1}{6}\left(\frac{1}{4M^{2}}\right)^{f(N)}.
\end{align*}
Thus, for large enough $n$ such that $n-1\ge f(N)^2$, we have
\begin{equation}\label{E9}
\begin{split}
\mu(J_{n})&\le\mu(J_{jf(n_{1})+(i-1)f(N)})\\ &\le\left|J_{jf(n_{1})+(i-1)f(N)}\right|^{s-\frac{15}{f(N)}-\frac{\log B}{\log 2}\cdot\frac{6(t+1)}{f(N)}}\\
&\le 6\cdot\left(4M^{2}\right)^{f(N)}\left|J_{n}\right|^{s-\frac{15}{f(N)}-\frac{\log B}{\log 2}\cdot\frac{6(t+1)}{f(N)}}\\&\ll6\left|J_{n}\right|^{s-\frac{15}{f(N)}-\frac{\log B}{\log 2}\cdot\frac{6(t+1)}{f(N)}-\frac{\log 4M^2}{\log 2}\cdot \frac{2}{f(N)}}.
\end{split}
\end{equation}
For the other case, there exists $1\le i\le v_1$ such that $$n_1f(n_1)+(i-1)f(N)<n<n_1f(n_1)+i f(N),$$ the process is similar as above, therefore we omit the details.

\item When $n=f(n_{k})-1$ for some $k\ge 2$, recall inequalities (\ref{c2}), one has for each $1\le\ell\le k-1$,
\begin{align*}
q_{f(n_{\ell+1})-f(n_{\ell})}&(w_{1}^{(\ell,1)},\ldots,w_{v_{\ell}}^{(\ell,n_{\ell})},2,\ldots,2)\le2^{2(f(n_{\ell+1})-m_{\ell})+3n_{\ell}-1}\cdot\\
&B^{dn_{\ell}^2}q_{v_{\ell}f(N)}(w_{1}^{(\ell,n_{\ell})},\ldots,w_{v_{\ell}}^{(\ell,n_{\ell})})\cdot\prod_{i=1}^{n_{\ell}-1}q_{f(n_{\ell})-1}(w_{1}^{(\ell,i)},\ldots,w_{r_{\ell}}^{(\ell,i)}).
\end{align*}
Our main concerns are to find  the coefficients of the product of convergents. For ease of notations, for any $2\le \ell\le k-1$,  we write
$$Q_{\ell}:=q_{v_{\ell-1}f(N)}(w_{1}^{(\ell-1,n_{\ell-1})},\ldots,w_{v_{\ell-1}}^{(\ell-1,n_{\ell-1})})\prod_{i=1}^{n_{\ell}-1}q_{f(n_{\ell})-1}(w_{1}^{(\ell,i)},\ldots,w_{r_{\ell}}^{(\ell,i)}).$$
 Thus, by the choice of largest $v_{\ell}$, one has $f(n_{\ell+1})-m_{\ell}\le f(N)$, then the length of $J_{f(n_k)-1}$ is
\begin{align*}
|J_{f(n_k)-1}|&\ge\frac{1}{2^5A_1(n_k)^{n_k}\prod_{\ell=1}^{k-1}2^{2(2f(N)+3n_{\ell})}B^{2dn_{\ell}^2}}\cdot\frac{1}{q_{v_{k-1}f(N)}^2(w_{1}^{(k-1,n_{k-1})},\ldots,w_{v_{k-1}}^{(k-1,n_{k-1})})}\cdot\\
&\quad\quad\quad\quad\quad\quad\quad\quad\quad\quad\quad\quad\prod_{i=0}^{n_{1}-1}\frac{1}{q_{f(n_1)-1}^{2}(w_{1}^{(1,i)},\ldots,w_{r_{1}}^{(1,i)})}\cdot\prod_{\ell=2}^{k-1}\frac{1}{Q_{\ell}^2}.
\end{align*}
Therefore, we have
\begin{equation*}\label{E10}
\begin{split}
&\quad\mu(J_{f(n_k)-1})=\mu(J_{m_{k-1}})\\
&\le \frac{1}{B^{dn_{k-1}^2}}\prod_{t=1}^{v_{k-1}}\frac{1}{q_{f(N)}^{2s}(w_t^{(k-1,n_{k-1})})B^{(2s-1)dN}}\cdot\prod_{i=1}^{n_{k-1}-1}\prod_{t=1}^{r_{k-1}}\frac{1}{q_{f(N)}^{2s}(w_t^{(k-1,i)})B^{(2s-1)dN}}\cdot\mu(J_{f(n_{k-1})-1})\\
&\le \prod_{\ell=1}^{k-1}\frac{1}{B^{dn_{\ell}^2}}\prod_{t=1}^{v_{\ell}}\frac{1}{q_{f(N)}^{2s}(w_t^{(\ell,n_{\ell})})B^{(2s-1)dN}}\cdot\prod_{i=1}^{n_{\ell}-1}\prod_{t=1}^{r_{\ell}}\frac{1}{q_{f(N)}^{2s}(w_t^{(\ell,i)})B^{(2s-1)dN}}\cdot\mu(J_{f(n_{1})-1}).
\end{split}
\end{equation*}
We split the calculations into several parts.
\begin{itemize}
\item Let
\begin{align*}
P_{1}:&=\frac{1}{B^{dn_{1}^2}}\prod_{i=0}^{n_{1}-1}\prod_{t=1}^{r_{1}}\frac{1}{q_{f(N)}^{2s}(w_t^{(1,i)})B^{(2s-1)dN}}\\
&=\frac{1}{B^{dn_{1}^2+(2s-1)dNr_1n_1}}\prod_{i=0}^{n_{1}-1}\left(\frac{1}{q_{f(n_1)-1}^{2}(w_1^{(1,i)},\ldots,w_{r_1}^{(1,i)})}\right)^{s-\frac{2}{f(N)}},
\end{align*}
then, we have
\begin{equation}\label{E101}
\begin{split}
P_1B^{2sdn_1^2}&=B^{(2s-1)n_1(r_1t-t+1)}\prod_{i=0}^{n_{1}-1}\left(\frac{1}{q_{f(n_1)-1}^{2}(w_1^{(1,i)},\ldots,w_{r_1}^{(1,i)})}\right)^{s-\frac{2}{f(N)}}\\
&\le B^{\frac{2(t+1)dn_1^2}{f(N)}}\prod_{i=0}^{n_{1}-1}\left(\frac{1}{q_{f(n_1)-1}^{2}(w_1^{(1,i)},\ldots,w_{r_1}^{(1,i)})}\right)^{s-\frac{2}{f(N)}}.
\end{split}
\end{equation}

\item For each $2\le \ell\le k-1$,
\begin{align*}
P_{\ell}:&=\frac{1}{B^{dn_{\ell}^2}}\prod_{t=1}^{v_{\ell-1}}\frac{1}{q_{f(N)}^{2s}(w_t^{(\ell-1,i)})B^{(2s-1)dN}}\cdot\prod_{i=1}^{n_{\ell}-1}\prod_{t=1}^{r_{\ell}}\frac{1}{q_{f(N)}^{2s}(w_t^{(\ell,i)})B^{(2s-1)dN}}\\
&=\frac{1}{B^{dn_{\ell}^2+(2s-1)dN(v_{\ell-1}+r_{\ell}(n_{\ell}-1))}}\cdot\left(\frac{1}{Q_{\ell}}\right)^{s-\frac{2}{f(N)}}.
\end{align*}
Recall that $n_{\ell}f(n_{\ell})+v_{\ell}f(N)< f(n_{\ell+1})\le n_{\ell}f(n_{\ell})+(v_{\ell}+1)f(N)$ and $n_{\ell}f(n_{\ell})\ll\varepsilon f(n_{\ell+1})$ by the choice of $r_{\ell}$ and $v_{\ell}$, these induce following inequalities that we will use later,
\begin{equation}\label{E10all}
\begin{split}
r_{\ell}dN&=f(n_{\ell})-1-r_{\ell}t,\\
v_{\ell-1}dN&\ge(1-\varepsilon)f(n_{\ell})-f(N)-v_{\ell-1}t,\\
n_{\ell}r_{\ell}t&=\frac{(dn_{\ell}+t-1)n_{\ell}t}{f(N)}\le\frac{2(t+1)dn_{\ell}^2}{f(N)},\\
\varepsilon f(n_{\ell})+v_{\ell-1}t&\le\frac{(1+\varepsilon)f(n_{\ell})}{f(N)}+\frac{f(n_{\ell})}{f(N)}\le\frac{5d(t+1)n_{\ell}}{f(N)}.
\end{split}
\end{equation}
Therefore,
\begin{equation}\label{E102}
\begin{split}
P_{\ell}B^{2sdn_{\ell}^2}&\le B^{\varepsilon f(n_{\ell})+f(N)+v_{\ell-1}t+n_{\ell}+n_{\ell}r_{\ell}t}\cdot\left(\frac{1}{Q_{\ell}}\right)^{s-\frac{2}{f(N)}}\\
&\le B^{\frac{2(t+1)dn_{\ell}^2}{f(N)}+\frac{5d(t+1)n_{\ell}}{f(N)}+f(N)+n_{\ell}} \cdot\left(\frac{1}{Q_{\ell}}\right)^{s-\frac{2}{f(N)}}.
\end{split}
\end{equation}
\item For the last part, let
\begin{align*}
P_{k}:&=\prod_{t=1}^{v_{k-1}}\frac{1}{q_{f(N)}^{2s}(w_{t}^{(k-1,n_{k-1})})B^{(2s-1)dN}}\\
&=\frac{1}{B^{(2s-1)dNv_{k-1}}}\left(\frac{1}{q_{v_{k-1}f(N)}^2(w_{1}^{(k-1,n_{k-1})},\ldots,w_{v_{k-1}}^{(k-1,n_{k-1})})}\right)^{s-\frac{2}{f(N)}}.
\end{align*}
Thus, we have
\begin{equation}\label{E103}
\begin{split}
A_1(n_k)^{sn_k}P_k&\le B^{\varepsilon f(n_k)+\varepsilon dn_k+f(N)+v_{k-1}t}\left(\frac{1}{q_{v_{k-1}f(N)}^2(w_{1}^{(k-1,n_{k-1})},\ldots,w_{v_{k-1}}^{(k-1,n_{k-1})})}\right)^{s-\frac{2}{f(N)}}\\
&\le B^{\frac{6d(t+1)n_k}{f(N)}+f(N)}\left(\frac{1}{q_{v_{k-1}f(N)}^2(w_{1}^{(k-1,n_{k-1})},\ldots,w_{v_{k-1}}^{(k-1,n_{k-1})})}\right)^{s-\frac{2}{f(N)}}.
\end{split}
\end{equation}
\end{itemize}
Take $j=1,n_1,\ldots,n_{k-1}$ respectively in (\ref{E31}) (\ref{E32}) and combined with (\ref{E101}) (\ref{E102}) (\ref{E103}), we have
\begin{align*}
\mu(J_{f(n_k)-1})&\le 2^{5+4(k-1)f(N)+(6+\frac{\log B}{\log 2})(n_1+\cdots+n_{k-1})}|J_{f(n_k)-1}|^{s-\frac{2}{f(N)}-\frac{\log B}{\log 2}\cdot\frac{16(t+1)}{f(N)}-\frac{\log B}{\log 2}\cdot\frac{2}{f(N)}}\\
&\le 2^5|J_{f(n_k)-1}|^{s-\frac{22}{f(N)}-\frac{\log B}{\log 2}\cdot\frac{16(t+1)}{f(N)}-\frac{\log B}{\log 2}\cdot\frac{4}{f(N)}}.
\end{align*}
For the remaining cases, it suffices to show the estimates for different forms of $P_k$. We write these as $P_{k}^{(j)}$ and $P_{k}^{\prime(j)}$.
\item When $n=jf(n_k)$ for some $k\ge 2$ and $1\le j\le n_k$, we first discuss a modified form of $P_k$.
\begin{align*}
P_{k}^{(j)}:&=\frac{1}{A_{j}(n_k)^{n_k}}\prod_{i=1}^{j-1}\frac{1}{A_{i}(n_k)^{n_{k}}}\prod_{t=1}^{r_{k}}\frac{1}{q_{f(N)}^{2s}(w_t^{(k,i)})B^{(2s-1)dN}}\cdot\prod_{t=1}^{v_{k-1}}\frac{1}{q_{f(N)}^{2s}(w_t^{(k-1,n_{k-1})})B^{(2s-1)dN}}\\
&=\frac{1}{(A_1(n_k)\cdots A_j(n_k))^{n_k}}\frac{1}{B^{(2s-1)dN(r_k(j-1)+v_{k-1})}}\cdot\\
&\quad\prod_{i=1}^{j-1}\left(\frac{1}{q_{f(n_k)-1}(w_{1}^{(k,i)},\ldots,w_{r_k}^{(k,i)})}\right)^{s-\frac{2}{f(N)}}\left(\frac{1}{q_{v_{k-1}f(N)}^2(w_{1}^{(k-1,n_{k-1})},\ldots,w_{r_k}^{(k-1,n_{k-1})})}\right)^{s-\frac{2}{f(N)}}\\
&=\frac{A_1(n_k)^{sjn_k}}{A_j(n_k)^{(1-s)n_k}(A_1(n_k)\cdots A_j(n_k))^{2sn_k}}\frac{1}{B^{(2s-1)dN(r_k(j-1)+v_{k-1})}}\cdot\\
&\quad\prod_{i=1}^{j-1}\left(\frac{1}{q_{f(n_k)-1}(w_{1}^{(k,i)},\ldots,w_{r_k}^{(k,i)})}\right)^{s-\frac{2}{f(N)}}\left(\frac{1}{q_{v_{k-1}f(N)}^2(w_{1}^{(k-1,n_{k-1})},\ldots,w_{r_k}^{(k-1,n_{k-1})})}\right)^{s-\frac{2}{f(N)}}.\\
\end{align*}
Thus, using inequalities in (\ref{E10all}), one concludes that
\begin{align*}
\quad(&A_1(n_k)\cdots  A_j(n_k))^{2sn_k} P_k^{(j)}\le B^{(2s-1+s\varepsilon)jdn_k-(2s-1)dN(r_k(j-1)+v_{k-1})}\cdot\\
&\quad\quad\prod_{i=1}^{j-1}\left(\frac{1}{q_{f(n_k)-1}(w_{1}^{(k,i)},\ldots,w_{r_k}^{(k,i)})}\right)^{s-\frac{2}{f(N)}}\left(\frac{1}{q_{v_{k-1}f(N)}^2(w_{1}^{(k-1,n_{k-1})},\ldots,w_{r_k}^{(k-1,n_{k-1})})}\right)^{s-\frac{2}{f(N)}}\\
&\le B^{3dn_k+\frac{2(t+1)dn_{k}^2}{f(N)}}\cdot\prod_{i=1}^{j-1}\left(\frac{1}{q_{f(n_k)-1}(w_{1}^{(k,i)},\ldots,w_{r_k}^{(k,i)})}\right)^{s-\frac{2}{f(N)}}\left(\frac{1}{q_{v_{k-1}f(N)}^2(w_{1}^{(k-1,n_{k-1})},\ldots,w_{r_k}^{(k-1,n_{k-1})})}\right)^{s-\frac{2}{f(N)}},
\end{align*}
where item $A_j(n_k)^{(1-s)n_k}$ is discarded in the first inequality. Recall the estimation of $q_{f(n_k)-1}$ in the previous case, the length of $J_{jf(n_k)}$ is
\begin{align*}
|J_{jf(n_k)}|&\ge \frac{1}{6q_{jf(n_k)}^{2}}\\
&\ge\frac{1}{6\cdot 2^{6j-2}(A_1(n_k)\cdots A_j(n_k))^{2n_k}\prod_{\ell=1}^{k-1}2^{2(2f(N)+3n_{\ell})}B^{2dn_{\ell}^2}}\cdot\prod_{i=1}^{j-1}\left(\frac{1}{q_{f(n_k)-1}^2(w_{1}^{(k,i)},\ldots,w_{r_k}^{(k,i)})}\right)\cdot\\
&\quad\frac{1}{q_{v_{k-1}f(N)}^2(w_{1}^{(k-1,n_{k-1})},\ldots,w_{v_{k-1}}^{(k-1,n_{k-1})})}\prod_{i=0}^{n_{1}-1}\frac{1}{q_{f(n_1)-1}^{2}(w_{1}^{(1,i)},\ldots,w_{r_1}^{(1,i)})}\cdot\prod_{\ell=2}^{k-1}\frac{1}{Q_{\ell}^2}.
\end{align*}
Therefore, take $j=1,n_1,\ldots,n_{k-1}$ respectively (\ref{E31}) (\ref{E32}) and combined with (\ref{E101}) (\ref{E102}), we have
\begin{equation}\label{E11}
\begin{split}
\mu(J_{jf(n_k)})&\le 6\cdot 2^{6j-2+4(k-1)f(N)+6(n_1+\cdots+n_{k-1})}|J_{jf(n_k)}|^{s-\frac{2}{f(N)}-\frac{\log B}{\log 2}\cdot\frac{16(t+1)}{f(N)}-\frac{\log B}{\log 2}\cdot\frac{8}{f(N)}}\\
&\le 6\cdot|J_{jf(n_k)}|^{s-\frac{34}{f(N)}-\frac{\log B}{\log 2}\cdot\frac{16(t+1)}{f(N)}-\frac{\log B}{\log 2}\cdot\frac{8}{f(N)}}.
\end{split}
\end{equation}

\item When $n=jf(n_k)-1$ for some $k\ge 2$ and $2\le j\le n_k$, we have
\begin{align*}
P_{k}^{\prime(j)}:&=\prod_{i=1}^{j-1}\frac{1}{A_{i}(n_k)^{n_{k}}}\prod_{t=1}^{r_{k}}\frac{1}{q_{f(N)}^{2s}(w_t^{(k,i)})B^{(2s-1)dN}}\cdot\prod_{t=1}^{v_{k-1}}\frac{1}{q_{f(N)}^{2s}(w_t^{(k-1,n_{k-1})})B^{(2s-1)dN}}\\
&=\frac{1}{(A_1(n_k)\cdots A_{j-1}(n_k))^{n_k}}\frac{1}{B^{(2s-1)dN(r_k(j-1)+v_{k-1})}}\cdot\\
&\quad\prod_{i=1}^{j-1}\left(\frac{1}{q_{f(n_k)-1}(w_{1}^{(k,i)},\ldots,w_{r_k}^{(k,i)})}\right)^{s-\frac{2}{f(N)}}\left(\frac{1}{q_{v_{k-1}f(N)}^2(w_{1}^{(k-1,n_{k-1})},\ldots,w_{r_k}^{(k-1,n_{k-1})})}\right)^{s-\frac{2}{f(N)}}\\
&=\frac{A_1(n_k)^{s(j-1)n_k}}{A_{j-1}(n_k)^{(1-s)n_k}(A_1(n_k)\cdots A_{j-1}(n_k))^{2sn_k}}\frac{1}{B^{(2s-1)dN(r_k(j-1)+v_{k-1})}}\cdot\\
&\quad\prod_{i=1}^{j-1}\left(\frac{1}{q_{f(n_k)-1}(w_{1}^{(k,i)},\ldots,w_{r_k}^{(k,i)})}\right)^{s-\frac{2}{f(N)}}\left(\frac{1}{q_{v_{k-1}f(N)}^2(w_{1}^{(k-1,n_{k-1})},\ldots,w_{r_k}^{(k-1,n_{k-1})})}\right)^{s-\frac{2}{f(N)}}.\\
\end{align*}
The coefficient in front of the product of convergents can be alternated by (\ref{cor2}), that is
$$A_{j}(n_k)^s=A_{1}(n_k)^sA_{j-1}(n_k)^{1-s}.$$
This implies that
\begin{align*}
&\quad A_j(n_k)^{sn_k}(A_1(n_k)\cdots  A_{j-1}(n_k))^{2sn_k} P_k^{\prime(j)}\le A_1(n_k)^{sjn_k}B^{-(2s-1)dN(r_k(j-1)+v_{k-1})}\cdot\\
&\quad\quad\quad\quad\quad\quad\quad\quad\prod_{i=1}^{j-1}\left(\frac{1}{q_{f(n_k)-1}(w_{1}^{(k,i)},\ldots,w_{r_k}^{(k,i)})}\right)^{s-\frac{2}{f(N)}}\left(\frac{1}{q_{v_{k-1}f(N)}^2(w_{1}^{(k-1,n_{k-1})},\ldots,w_{r_k}^{(k-1,n_{k-1})})}\right)^{s-\frac{2}{f(N)}}\\
&\le B^{3dn_k+\frac{2(t+1)dn_{k}^2}{f(N)}}\cdot\quad\prod_{i=1}^{j-1}\left(\frac{1}{q_{f(n_k)-1}(w_{1}^{(k,i)},\ldots,w_{r_k}^{(k,i)})}\right)^{s-\frac{2}{f(N)}}\left(\frac{1}{q_{v_{k-1}f(N)}^2(w_{1}^{(k-1,n_{k-1})},\ldots,w_{r_k}^{(k-1,n_{k-1})})}\right)^{s-\frac{2}{f(N)}},
\end{align*}
which is the same as in the previous case. The length of $J_{jf(n_k)-1}$ is given by
\begin{align*}
|J_{jf(n_k)-1}|&\ge \frac{1}{2^5A_j(n_k)^{n_k}q_{jf(n_k)-1}^{2}}\\
&\ge \frac{1}{2^{6j-1}A_j(n_k)^{n_k}(A_1(n_k)\cdots A_{j-1}(n_k))^{2n_k}}\cdot\prod_{i=1}^{j-1}\left(\frac{1}{q_{f(n_k)-1}^2(w_1^{(k,i)},\ldots,w_{r_k}^{(k,i)})}\right)\frac{1}{q_{f(n_k)-1}^2}\\
&\ge\frac{1}{2^{6j-1}A_j(n_k)^{n_k}(A_1(n_k)\cdots A_{j-1}(n_k))^{2n_k}\prod_{\ell=1}^{k-1}2^{2(2f(N)+3n_{\ell})}B^{2dn_{\ell}^2}}\prod_{i=1}^{j-1}\left(\frac{1}{q_{f(n_k)-1}^2(w_1^{(k,i)},\ldots,w_{r_k}^{(k,i)})}\right)\cdot\\
&\quad\frac{1}{q_{v_{k-1}f(N)}^2(w_{1}^{(k-1,n_{k-1})},\ldots,w_{v_{k-1}}^{(k-1,n_{k-1})})}\prod_{i=0}^{n_{1}-1}\frac{1}{q_{f(n_1)-1}^{2}(w_{1}^{(1,i)},\ldots,w_{r_1}^{(1,i)})}\cdot\prod_{\ell=2}^{k-1}\frac{1}{Q_{\ell}^2}.
\end{align*}
Therefore, similar with (\ref{E31}) (\ref{E32}) and combined with (\ref{E101}) (\ref{E102}), we have
\begin{equation}\label{E11}
\begin{split}
\mu(J_{jf(n_k)-1})&\le 6\cdot 2^{6j-1+4(k-1)f(N)+6(n_1+\cdots+n_{k-1})}|J_{jf(n_k)-1}|^{s-\frac{2}{f(N)}-\frac{\log B}{\log 2}\cdot\frac{16(t+1)}{f(N)}-\frac{\log B}{\log 2}\cdot\frac{8}{f(N)}}\\
&\le 6\cdot|J_{jf(n_k)-1}|^{s-\frac{34}{f(N)}-\frac{\log B}{\log 2}\cdot\frac{16(t+1)}{f(N)}-\frac{\log B}{\log 2}\cdot\frac{8}{f(N)}}.
\end{split}
\end{equation}

\item When $n=jf(n_k)+if(N)$ for some $1\le j< n_k$ and $1\le i<r_k$, we have
\begin{equation}\label{E12}
\begin{split}
\mu(J_{n})&\le \prod_{t=1}^{i}\frac{1}{q_{f(N)}^{2s}(w_t^{(k,j)})B^{(2s-1)dN}}\cdot\mu(J_{jf(n_k)})\\
&\le \left(\frac{1}{q_{if(N)}^2(w_1^{(k,j)},\ldots,w_i^{(k,j)})}\right)^{s-\frac{2}{f(N)}}\cdot\left(\frac{1}{q_{jf(n_k)}^2}\right)^{s-\frac{34}{f(N)}-\frac{\log B}{\log 2}\cdot\frac{16(t+1)}{f(N)}-\frac{\log B}{\log 2}\cdot\frac{8}{f(N)}}\\
&\le2\left(\frac{1}{q_{n}^2}\right)^{s-\frac{34}{f(N)}-\frac{\log B}{\log 2}\cdot\frac{16(t+1)}{f(N)}-\frac{\log B}{\log 2}\cdot\frac{8}{f(N)}}\\
&\le 12|J_n|^{s-\frac{34}{f(N)}-\frac{\log B}{\log 2}\cdot\frac{16(t+1)}{f(N)}-\frac{\log B}{\log 2}\cdot\frac{8}{f(N)}}.
\end{split}
\end{equation}
When $n=n_kf(n_k)+if(N)$ with $1\le i\le v_k$, the process is similar and the H\"older exponent is the same. We omit the details.

\item When $m_k+1\le n< f(n_{k+1})-1$, the gap between $n$ and $m_k$ is less than $f(N)$. This case is similar with $k=1$.
\begin{equation}\label{E13}
\begin{split}
\mu(J_{n})&=\mu(J_{m_k})\le \prod_{t=1}^{v_k}\frac{1}{q_{f(N)}^{2s}(w_t^{(k,n_k)})B^{(2s-1)dN}}\cdot\mu(J_{n_kf(n_k)})\\
&\le \left(\frac{1}{q_{v_kf(N)}^2(w_1^{(k,n_k)},\ldots,w_{v_k}^{(k,n_k)})}\right)^{s-\frac{2}{f(N)}}\cdot\left(\frac{1}{q_{n_kf(n_k)}^2}\right)^{s-\frac{34}{f(N)}-\frac{\log B}{\log 2}\cdot\frac{16(t+1)}{f(N)}-\frac{\log B}{\log 2}\cdot\frac{8}{f(N)}}\\
&\le2\left(\frac{1}{q_{m_k}^2}\right)^{s-\frac{34}{f(N)}-\frac{\log B}{\log 2}\cdot\frac{16(t+1)}{f(N)}-\frac{\log B}{\log 2}\cdot\frac{8}{f(N)}}\\
&\le2^{4(n-m_k)+6}\left(\frac{1}{q_{n}^2}\right)^{s-\frac{34}{f(N)}-\frac{\log B}{\log 2}\cdot\frac{16(t+1)}{f(N)}-\frac{\log B}{\log 2}\cdot\frac{8}{f(N)}}\\
&\le 6\cdot2^6|J_n|^{s-\frac{34}{f(N)}-\frac{\log B}{\log 2}\cdot\frac{16(t+1)}{f(N)}-\frac{\log B}{\log 2}\cdot\frac{8}{f(N)}-\frac{16d(t+1)}{f(N)}},
\end{split}
\end{equation}
where the last inequality holds by the same process as in the case $k=1$, that is
$$q_n^{2\cdot\frac{16d(t+1)}{f(N)}}\ge 2^{4(n-m_k)}.$$

\item For the remaining cases, when $n<f(n_{k+1})-1$, there are a total of  two cases that should be discussed. One case is that $jf(n_k)+(i-1)f(N)<n<jf(n_k)+i f(N)$ where $1\le i\le r_k$ and $0\le j< n_k$. The second case is that  there exists $1\le i\le v_1$ such that $n_kf(n_k)+(i-1)f(N)<n<n_kf(n_k)+i f(N)$. The estimates for both of them are similar to the case $k=1$, therefore, we omit the details.
\end{itemize}

Finally, in a nutshell,  for any $n\ge 1$, the mass on fundamental intervals is about
\begin{equation}\label{ET}
\mu(J_{n})\le c_1|J_{n}|^{s-\frac{c_2}{f(N)}},
\end{equation}
where $c_1,c_2$ are absolute constant independent with $N$.

\subsection{H\"older exponent for a general ball $B(x, r)$} Fix $x\in F$ and the radius of the ball $B(x, r)$ is small enough. There exists a unique sequence $a_1,a_2,\ldots$ such that $x\in J_m(a_1,\ldots,a_m)$ for any $m\ge 1$. Let $n$ be the integer such that
\begin{align*}
g_{n+1}\left(a_{1},\ldots,a_{n+1}\right)\le r<g_{n}\left(a_{1},\ldots,a_{n}\right).
\end{align*}
Then the ball $B(x,r)$ can only intersect one fundamental interval of order $n$, that is $J_{n}\left(a_{1},\ldots,a_{n}\right)$. Thus, all fundamental intervals of order $n+1$ which has non-empty intersection with $B(x,r)$ are contained in $J_{n}\left(a_{1},\ldots,a_{n}\right)$.
Therefore,
\begin{itemize}
\item When $n=jf(n_{k})-1$ for any $k\ge 1$ and $1\le j\le n_{k}$, the length of $I_{jf(n_k)}$ is
\begin{align*}
|I_{jf(n_{k})}|=\frac{1}{q_{jf(n_{k})}(q_{jf(n_{k})}+q_{jf(n_{k})-1})}\ge\frac{1}{2q_{jf(n_{k})}^{2}}\ge\frac{1}{2^5A_{j}(n_k)^{n_{k}}q_{jf(n_{k})-1}^{2}}.
\end{align*}
\begin{itemize}
\item When $$r<\frac{1}{2^5A_{j}(n_k)^{n_{k}}q_{jf(n_{k})-1}^{2}},$$ the ball $B\left(x,r\right)$ can intersect at most three basic intervals of order $jf(n_{k})$, all of these has same mass. Combined with $(\ref{gap})$ and $(\ref{ET})$ , we have
\begin{equation}
\begin{split}
\mu\left(B\left(x,r\right)\right)&\le3\mu(J_{jf(n_{k})})\le 3c_1|J_{jf(n_{k})}|^{s-\frac{c_2}{f(N)}}\le 3c_1M|g_{n+1}|^{s-\frac{c_2}{f(N)}}\le 3c_1Mr^{s-\frac{c_2}{f(N)}}.
\end{split}
\end{equation}
Here, $g_{n+1}$ represents the gap $g_{n+1}(a_1,\ldots,a_{n+1})$ for simplicity.
\item  When $$r>\frac{1}{2^5A_{j}(n_k)^{n_{k}}q_{jf(n_{k})-1}^{2}},$$ the number of basic intervals of order $jf(n_k)$ that intersect $B(x,r)$ is at most
\begin{align*}
2r\cdot 2^5A_{j}(n_k)^{n_{k}}q_{jf(n_{k})-1}^{2}+2\le 2^7rA_{j}(n_k)^{n_{k}}q_{jf(n_{k})-1}^{2}.
\end{align*}
Thus, by $\left(\ref{ET}\right)$, we have
\begin{equation}
\begin{split}
\mu\left(B\left(x,r\right)\right)&\le\min\Big\{\mu(J_{n}),2^7rA_{j}(n_k)^{n_{k}}q_{n}^{2}\mu(J_{n+1})\Big\}\\
&\le\mu(J_{n})\min\Big\{1,2^7rA_{j}(n_k)^{n_{k}}q_{n}^{2}\frac{1}{A_{j}(n_k)^{n_{k}}}\Big\}\\
&\le c_1|J_{n}|^{s-\frac{c_2}{f(N)}}\min\Big\{1,2^7rq_{n}^{2}\Big\}\\
&\le c_1\left(\frac{1}{q_{n}^{2}}\right)^{s-\frac{c_2}{f(N)}}\cdot (2^7rq_{n}^{2})^{s-\frac{c_2}{f(N)}}\\
&\le 2^7c_1r^{s-\frac{c_2}{f(N)}}.
\end{split}
\end{equation}
\end{itemize}


\item When $m_k\le n\le f(n_{k+1})-2$ for any $k\ge 1$, the length of basic interval $I_{n+1}$ is
$$|I_{n+1}|=\frac{1}{q_{n+1}\left(q_{n+1}+q_{n}\right)}\ge\frac{1}{2^5q_{n}^{2}}.$$
When $r$ is smaller than $|I_{n+1}|$, the situation is pretty much the same as the previous case. As for other situation, the number of basic intervals of order $n+1$ that intersect $B(x,r)$ is at most
\begin{align*}
2r\cdot 2^5q_{n}^{2}+2\le 2^7rq_{n}^{2}.
\end{align*}
Thus, by $(\ref{ET})$, we have
\begin{equation}
\begin{split}
\mu(B(x,r))&\le\min\Big\{\mu\left(J_{n}\right),2^7rq_{n}^{2}\mu(J_{n+1})\Big\}\\
&\le\mu(J_{n})\min\Big\{1,2^7rq_{n}^{2}\Big\}\\
&\le c_1|J_{n}|^{s-\frac{c_2}{f(N)}}\min\Big\{1,2^7rq_{n}^{2}\Big\}\\
&\le c_1\left(\frac{1}{q_{n}^{2}}\right)^{s-\frac{c_2}{f(N)}}\left(2^7rq_{n}^{2}\right)^{s-\frac{c_2}{f(N)}}\\
&\le 2^7c_1r^{s-\frac{c_2}{f(N)}}.
\end{split}
\end{equation}

\item For the remaining cases, as the construction in $F$, we know that $1\le a_{n}\le M$ for any $n\ge 1$. Then by $(\ref{ET})$ and $(\ref{gap})$, we have
\begin{align*}
\mu(B(x,r))&\le\mu(J_{n})\le c_1|J_{n}|^{s-\frac{c_2}{f(N)}}\le c_1\left(\frac{1}{q_{n}^{2}}\right)^{s-\frac{c_2}{f(N)}}\le 4c_1M^{2}\left(\frac{1}{q_{n+1}^{2}}\right)^{s-\frac{c_2}{f(N)}}\\
&\le 24c_1M^{2}\left|J_{n+1}\right|^{s-\frac{c_2}{f(N)}}\le 24c_1M^{3}\left|g_{n+1}\right|^{s-\frac{c_2}{f(N)}}\le 24c_1M^{3}r^{s-\frac{c_2}{f(N)}}.
\end{align*}
\end{itemize}

We conclude by the mass distribution principle (Proposition $\ref{MD}$)
\begin{align*}
\dim_\HH F\ge s-\frac{c_2}{f(N)}.
\end{align*}
Letting $N\rightarrow\infty$ and then $M\rightarrow\infty$ implies that
\begin{align*}
\dim_\HH E_B\ge \Big\{s:\textrm{P}(-s\log |T'(x)|-(2s-1)\log B)\le 0\Big\}.
\end{align*}

\section{Upper bound of $\dim_\HH E_B$}
For the upper bound of $\dim_\HH E_B$, we use the natural cover for the limsup set. Firstly, we quote a result from \cite[Lemma 4.8]{HuWuXu}, that we will use later.
\begin{lemma}[{\cite[Lemma 4.8]{HuWuXu}}]
Let $\phi$ be an arbitrary positive function defined on natural numbers $\mathbb{N}$, we have
\begin{equation}\label{lem31}
\sum_{1\le a_{f(n)}a_{2f(n)}\cdots a_{kf(n)}<\phi(n)}\left(\frac{1}{a_{f(n)}a_{2f(n)}\cdots a_{kf(n)}}\right)^s\ll \phi(n)^{1-s}(\log\phi(n))^{k-1}.
\end{equation}
\end{lemma}
For the coverings of the set $E_B$, we note that $E_B:=\limsup _{n\to\infty}E_n$, where 
\begin{align*}
E_{n}:&=\Big\{x\in\left[0,1\right):\prod_{i=1}^{n}a_{if(n)}\left(x\right)\ge B^{dn^{2}}\Big\}\\
&=\bigcup_{k=1}^{n}\Big\{x\in\left[0,1\right):\prod_{i=1}^{k}a_{if(n)}(x)\ge B^{dn^{2}}, \ \prod_{i=1}^{k}a_{if(n)}(x)<B^{dn^{2}}\Big\}:=\bigcup_{k=1}^{n}F_{n,k}.
\end{align*}
Then, fix an integer $1\le k\le n$, we consider the cover of $F_{n,k}$ as
\begin{align*}
F_{n,k}&=\bigcup_{\substack{a_{1},\ldots,a_{kf(n)}:\prod_{i=1}^{k}a_{if(n)}\ge B^{dn^{2}},\\ \prod_{i=1}^{k-1}a_{if(n)}<B^{dn^{2}}}}I_{kf(n)}\left(a_{1},\ldots, a_{kf(n)}\right)\\
&=\bigcup_{\substack{a_{1},\ldots,a_{kf(n)-1}:\\ \prod_{i=1}^{k-1}a_{if(n)}<B^{dn^2}}}\bigcup_{\substack{a_{kf(n)}: \\ \prod_{i=1}^{k}a_{if(n)}\ge B^{dn^{2}}}}I_{kf(n)}\left(a_{1},\ldots,a_{kf(n)}\right)\\
&=\bigcup_{\substack{a_1,\ldots,a_{kf(n)-1}:\\ \prod_{i=1}^{k-1}a_{if(n)}<B^{dn^2}}}J_{kf(n)-1}\left(a_{1},\ldots,a_{kf(n)-1}\right).
\end{align*}
Then the length of $J_{kf(n)-1}$ is given by
$$|J_{kf(n)-1}\left(a_1,\ldots,a_{kf(n)-1}\right)|\asymp \left(q_{kf(n)-1}^2\frac{B^{dn^2}}{\prod_{i=1}^{k-1}a_{if(n)}}\right)^{-1}.
$$
Thus the $s$-volume of this cover of $F_{n,k}$ is
\begin{equation*}
\begin{split}
\sum_{\substack{a_{1},\ldots,a_{kf(n)-1}: \\ \prod_{i=1}^{k-1}a_{if(n)}<B^{dn^{2}}}}\left|J_{kf(n)-1}\right|^{s}&\asymp\sum_{\substack{a_{1},\ldots,a_{kf(n)-1}:\\ \prod_{i=1}^{k-1}a_{if(n)}<B^{dn^{2}}}}\left(q_{kf(n)-1}^2\frac{B^{dn^{2}}}{\prod_{i=1}^{k-1}a_{if(n)}}\right)^{-s}\\
&\le \left(\sum_{a_{1},\ldots,a_{f(n)-1}}\frac{1}{q_{f(n)-1}^{2s}}\right)^{k}\sum_{\substack{a_{f(n)}, \ldots, a_{(k-1)f(n)}:\\\prod_{i=1}^{k-1}a_{if(n)}<B^{dn^2}}}\left(\frac{1}{a_{f(n)}a_{2f(n)}\cdots a_{(k-1)f(n)}}\right)^s\frac{1}{B^{dn^{2}s}}\\
&\ll n^{2(k-1)}(d\log B)^{k-1} \left(\sum_{a_{1},\ldots,a_{f(n)-1}}\frac{1}{q_{f(n)-1}^{2s}}\right)^{k}B^{dn^2\left(1-s\right)}\frac{1}{B^{dn^2s}},
\end{split}
\end{equation*}
where the last inequality holds by taking $\phi(n)=B^{dn^{2}}$ in (\ref{lem31}). Therefore, the $s$-volume of a cover of $E_n$ is bounded by
\begin{align*}
&\sum_{k=1}^nn^{2(k-1)}(d\log B)^{k-1} \left(\sum_{a_{1},\ldots,a_{f(n)-1}}\frac{1}{q_{f(n)-1}^{2s}}\right)^{k-1}B^{dn^2\left(1-s\right)}\frac{1}{B^{dn^2s}}\\
&\quad\quad\quad\quad\quad\quad\quad\quad\quad\ll n^{2n-1}(d\log B)^{n-1}\left(\sum_{a_{1},\ldots,a_{f(n)-1}}\frac{1}{q_{f(n)-1}^{2s}}\right)^{n}\frac{1}{B^{dn^{2}(2s-1)}}.
\end{align*}
This implies that
\begin{equation}
\begin{split}
\mathcal{H}^{s}\left(E_B\right)&\le\liminf_{N\rightarrow\infty}\sum_{n\ge N}\sum_{1\le k\le n}\sum_{\substack{a_{1},\ldots,a_{kf(n)-1}: \\ \prod_{i=1}^{k-1}a_{if(n)}<B^{dn^2}}}\left|J_{kf(n)-1}\left(a_{1},\ldots,a_{kf(n)-1}\right)\right|^{s}\\
&\le\liminf_{N\rightarrow\infty}\sum_{n\ge N}n^{2n-1}(d\log B)^{n-1}\left(\sum_{a_{1},\ldots,a_{f(n)-1}}\frac{1}{q_{f(n)-1}^{2s}}\right)^{n}\frac{1}{B^{dn^{2}(2s-1)}}.
\end{split}
\end{equation}
\medskip

Finally, to complete the proof, we claim that the dimensional numbers
$$s_0^\prime:=\inf\left\{s: \sum_{n\ge N}\left(\sum_{a_{1},\ldots,a_{f(n)-1}}\frac{1}{q_{f(n)-1}^{2s}}\right)^{n}\frac{1}{B^{dn^{2}\left(2s-1\right)}}<\infty\right\}$$
and
$$s_0:=\inf\left\{s: \lim_{n\to\infty}\frac{1}{dn}\log\left(\sum_{a_{1},\ldots,a_{f(n)-1}}\frac{1}{q_{f(n)-1}^{2s}}\right)\frac{1}{B^{dn\left(2s-1\right)}}\leq 0\right\}$$
are the same.
\begin{proposition}With the above notations $s_0=s_0^\prime$.
\end{proposition}

\begin{proof} We first show that $s_0\geq s_0^\prime$. To this end, for any $\epsilon>0$, let $s^{\prime}:=s_0+2\epsilon$. For each $s^{\prime}>s_0$, we consider $s=(s_0+s^{\prime})/2$, {then we have}
$$\frac1n\log\left(\sum_{a_{1},\ldots,a_{f(n)-1}}\frac{1}{q_{f(n)-1}^{2\cdot\frac{s^{\prime}+s_0}{2}}}\right)\frac{1}{B^{dn\left(2\cdot\frac{s^{\prime}+s_0}{2}-1\right)}}\leq \frac{d\epsilon}{\alpha},$$
where $\alpha>\frac1{2\log B}$ is a constant. This implies that
\begin{align*}
&\sum_{a_1,\ldots, a_{f(n)-1}}\frac1{q_{f(n)-1}^{2(s^{\prime}-\epsilon)}}\frac1{B^{dn(2(s^{\prime}-\epsilon)-1)}}\leq e^{n\frac{d\epsilon}{\alpha}} \quad \Longrightarrow\quad \sum_{a_1,\ldots, a_{f(n)-1}}\frac1{q_{f(n)-1}^{2s^{\prime}}}\frac1{B^{dn(2s^{\prime}-1)}}\leq \frac{e^{\frac{dn\epsilon}{\alpha}}}{B^{2dn\epsilon}}.
\end{align*}
By the choice of $\alpha$ we can conclude that
$$\sum_{n=N}^\infty\left(\sum_{a_1,\ldots, a_{f(n)-1}}\frac1{q_{f(n)-1}^{2s^{\prime}}}\frac1{B^{dn(2s^{\prime}-1)}}\right)^n<\infty,$$
and $s^{\prime}>s_0^\prime.$ Therefore, $s_0\geq s_0^\prime$.

For the reverse inequality, for all $s<s_0$, there exists $\eta>0$ such that
$$\lim_{n\to\infty}\frac{1}{dn}\log\left(\sum_{a_{1},\ldots,a_{f(n)-1}}\frac{1}{q_{f(n)-1}^{2s}}\right)\frac{1}{B^{dn\left(2s-1\right)}}>\eta.$$
That is to say for large enough $n$, we have
$$\sum_{a_{1},\ldots,a_{f(n)-1}}\frac{1}{q_{f(n)-1}^{2s}}\frac{1}{B^{dn\left(2s-1\right)}}>e^{\eta dn}.$$ Thus,
$$\sum_{n=N}^\infty\left(\sum_{a_{1},\ldots,a_{f(n)-1}}\frac{1}{q_{f(n)-1}^{2s}}\frac{1}{B^{dn\left(2s-1\right)}}\right)^n>\sum_{n=N}^\infty e^{\eta dn^2}=\infty.$$ Hence $s<s_0^\prime$ and as a result $s_0\leq s_0^\prime$.
\end{proof}
Therefore, the upper bound of the Hausdorff dimension of $E_B$ is as follows
$$\dim_\HH E_B\le \inf \Big\{s:\textrm{P}(-s\log |T'(x)|-(2s-1)\log B)\le 0\Big\}.$$

\section{Hausdorff dimension of $E_f(\psi)$}
In this section we focus on the dimension of $E_f\left(\psi\right)$ for a general function $\psi(n)$. First, we quote a well-known result of L\"uczak \cite{Luczak} that will be used later.
\begin{lemma}\label{41}
For any $b,c>1$, $\Big\{x\in[0,1):a_{n}\left(x\right)\ge c^{b^{n}}\textrm{ for infinitely many }n\Big\}$ and $\Big\{x\in[0,1):a_{n}\left(x\right)\ge c^{b^{n}}\textrm{ for sufficiently large }n\Big\}$ have the same Hausdorff dimension $1/\left(b+1\right)$.
\end{lemma}
Recall that
\begin{equation*}
E_{f}(\psi):=\left\{x\in [0, 1): \sqrt[n]{a_{f(n)}(x)a_{2f(n)}(x)\cdots a_{nf(n)}(x)}\geq \psi(n) \ {\rm for \ infinitely \ many} \ n\in \N\right\},
\end{equation*}
where $$\log B=\liminf_{n\rightarrow\infty}\frac{\log\psi\left(n\right)}{dn},\textrm{ }\log b=\liminf_{n\rightarrow\infty}\frac{\log\log\psi\left(n\right)}{dn^2}.$$
The proof is split into several parts.
\begin{itemize}
\item For the case $B=1$, we observe that
$$
E_f\left(\psi\right)\supseteq\Big\{x\in\left[0,1\right):a_{nf(n)}\left(x\right)\ge \psi\left(n\right)^n\textrm{ for infinitely many }n\in\mathbb{N}\Big\}.
$$
By Wang-Wu theorem \cite{WaWu08}, we obtain that $\dim_\HH E_f\left(\psi\right)=1$.

\item For the case $1<B<\infty$, for any $0<\varepsilon<B-1$ we have $\psi\left(n\right)>\left(B-\varepsilon\right)^{dn}$ when $n$ is large enough. Thus, we have
$$
E_f\left(\psi\right)\subseteq\Big\{x\in\left[0,1\right):\prod_{i=1}^{n}a_{if(n)}\left(x\right)\ge\left(B-\varepsilon\right)^{dn^{2}}\textrm{ for infinitely many }n\in\mathbb{N}\Big\}.
$$
On the other hand, we can choose a largely sparse integer sequence $\{n_j\}_{j\ge 1}$ such that for all $j\ge 1$,
$$\psi(n_j)<\left(B+\varepsilon\right)^{dn_j}.$$ Choose large enough integers $N,M$, let $s<s_{B+\varepsilon}(M,N)$, similar as we discussed in Section 4, the difference in this case is $f(n_j)-1$ may not be a multiple of $f(N)$. So we define $\hat{r}_j$ such that
$$\hat{r}_jf(N)\le f(n_j)-1<(\hat{r}_j+1)f(N).$$
For any $1\le i\le n_j-1$, let
$$o_j^i:=if(n_j)+\hat{r}_jf(N),$$ one can see that $0\le (i+1)f(n_j)-o_j^i<f(N)$. Let $\hat{v}_j$ be the largest integer such that $n_{j}f(n_{j})+\hat{v}_{j}f(N)\le f(n_{j+1})-1$ and denote $\hat{m}_{j}=n_{j}f(n_{j})+v_{j}f(N)$. Then we define a subset of $E_f(\psi)$,
\begin{align*}
\hat{F}:=\Big\{x\in\left[0,1\right)&: A_{i}(n_j)^{n_j}\le a_{if(n_j)}(x)\le 2A_{i}(n_j)^{n_j}, 1\le i\le n_j,\\
&\phantom{=\;\;}a_n(x)=2 \textrm{ for all } \hat{m}_j<n<f(n_{j+1}),j\ge 1,\\
&\phantom{=\;\;}\textrm{ and } o_j^i<n<(i+1)f(n_j),1\le i<n_j;1\le a_{n}\left(x\right)\le M\textrm{ for other cases}\Big\}.
\end{align*}
A rigorous proof of dimension of this set can be carried out with no more changes, one can conclude that
$$
s_{B+\varepsilon}\le\dim_\HH E_f\left(\psi\right)\le s_{B-\varepsilon}.
$$
Letting $\varepsilon\rightarrow 0$, we obtain $\dim_\HH E_f\left(\psi\right)=s_{B}$.
\item For the case $B=\infty$, we consider the following three cases.
\begin{enumerate}[(i)]
\item For $1<b<\infty$, for any $0<\varepsilon<b-1$ we have $\psi\left(n\right)>e^{\left(b-\varepsilon\right)^{(1-\varepsilon) nf(n)}}$ for sufficiently large $n$. We observe that $$\prod_{i=1}^{n}a_{if(n)}\left(x\right)\ge e^{n(b-\varepsilon)^{(1-\varepsilon)nf(n)}}$$ implies that there is an $1\le i\le n$ such that $$a_{if(n)}\left(x\right)\ge e^{(b-\varepsilon)^{(1-\varepsilon)if(n)}}$$ for any $x\in E_f\left(\psi\right)$. Thus,
$$E_f\left(\psi\right)\subseteq\Big\{x\in\left[0,1\right):a_{n}\left(x\right)\ge e^{\left(b-\varepsilon\right)^{(1-\varepsilon)n}}\textrm{ for infinitely many }n\in\mathbb{N}\Big\}.$$
Therefore, by Lemma \ref{41} we have
$$\dim_\HH E_f\left(\psi\right)\le\frac{1}{1+(b-\varepsilon)^{1-\varepsilon}}.$$
Since $\varepsilon>0$ is arbitrary, we have
$$\dim_\HH E_f\left(\psi\right)\le\frac{1}{1+b}.$$
As for the lower bound, we observe that for any $\varepsilon>0$, we have $\psi\left(n\right)<e^{\left(b+\varepsilon\right)^{(1+\varepsilon)nf(n)}}$ for infinitely many $n$. Let
$$\mathcal{G}_{b}=\Big\{n:\psi\left(n\right)<e^{\left(b+\varepsilon\right)^{(1+\varepsilon)nf(n)}}\Big\}.$$
Then we have
\begin{align*}
E_f\left(\psi\right)&\supseteq\Big\{x\in\left[0,1\right):\prod_{i=1}^{n}a_{if(n)}\left(x\right)\ge e^{\left(b+2\varepsilon\right)^{(1+\varepsilon)nf(n)}}\textrm{ for infinitely many }n\in\mathcal{G}_{b}\Big\}\\
&\supseteq\Big\{x\in\left[0,1\right):a_{if(n)}\left(x\right)\ge e^{\left(b+2\varepsilon\right)^{(1+\varepsilon)if(n)}},1\le i\le n,\textrm{ for infinitely many }n\in\mathcal{G}_{b}\Big\}\\
&\supseteq\Big\{x\in\left[0,1\right):a_{n}\left(x\right)\ge e^{\left(b+2\varepsilon\right)^{(1+\varepsilon)n}},\textrm{ for all }n\in\mathbb{N}\Big\}.
\end{align*}
Thus, by Lemma \ref{41} and since $\varepsilon>0$ is arbitrary, we have
$$\dim_\HH E_f\left(\psi\right)=\frac{1}{1+b}.$$
\item For $b=1$, the proof of the lower bound is the same as in the case $1<b<\infty$, take $b=1$ implies that $\dim_\HH E_f\left(\psi\right)\ge 1/2$. As for the upper bound, using the fact that $B=\infty$, we have
$$\frac{\log\psi\left(n\right)}{dn}>\log B_{1}$$
for a sufficiently large $B_{1}>0$. This implies $\psi\left(n\right)>B_{1}^{dn}$ holds for sufficiently large $n$. Then, we have
$$E_f\left(\psi\right)\subseteq\Big\{x\in\left[0,1\right):\prod_{i=1}^{n}a_{if(n)}\left(x\right)\ge B_{1}^{dn^{2}}\textrm{ for infinitely many }n\in\mathbb{N}\Big\}.$$
Thus,
$$\dim_\HH E_f\left(\psi\right)\le s_{B_{1}}.$$
Letting $B_{1}\rightarrow\infty$, we obtain the desired result.
\item For $b=\infty$, for a sufficiently large $b_1$ such that $0<b_{1}<b$, by the first part in (i), we have
$$\dim_\HH E\left(\psi\right)\le\frac{1}{b_1+1}\rightarrow 0$$
as $b_1\rightarrow\infty$.
\end{enumerate}
We have covered all potential cases, thus the theorem is proven.
\end{itemize}

\end{document}